\documentclass[11pt]{article}
\usepackage{amsmath,amssymb,amsthm,amscd,}
\usepackage{epic}
\usepackage{epsfig}
\begin{document}

\newtheorem{thm}{Theorem}[section]
 \newtheorem{cor}[thm]{Corollary}
 \newtheorem{lem}[thm]{Lemma}
 \newtheorem{prop}[thm]{Proposition}
\newtheorem{claim}[thm]{Claim}
 \newtheorem{df}[thm]{Definition}
 \newtheorem{rem}[thm]{Remark}
 \newtheorem{ex}{Example}
 \numberwithin{equation}{section}

\newcommand\<{\langle}
\renewcommand\>{\rangle}

\author{P. Germain}

\title{Global existence for coupled Klein-Gordon equations with different speeds}

\maketitle

\begin{abstract}
Consider, in dimension 3, a system of coupled Klein-Gordon equations with different speeds, and an arbitrary quadratic nonlinearity. We show, for data which are small, smooth, and localized, that a global solution exists, and that it scatters. The proof relies on the space-time resonance approach; it turns out that the resonant structure of this equation has features which were not studied before, but which are generic in some sense.
\end{abstract}

\tableofcontents

\section{Introduction}

\subsection{Presentation of the problem}

The aim of this paper is to prove global existence and scattering for the system of Klein-Gordon equations
\begin{equation}
\label{mandarin}
\left\{ \begin{array}{l} \square u^1 + u^1 = Q^1(u^1,u^c) \\
\square_c u^c + u^c = Q^c(u^1,u^c) \\
(u^1,\partial_t u^1)(t=0) = (u^1_0,u^1_1) \\ (u^c,\partial_t u^c)(t=0) = (u^c_0,u^c_1) \\ \end{array} \right.
\end{equation}
where $u^1$, $u^c$, are real functions of $(t,x) \in \mathbb{R}\times \mathbb{R}^3$, we denoted
$$
\square \overset{def}{=} \partial_t^2 - \Delta\;\;\;\mbox{and}\;\;\;\square_c \overset{def}{=} \partial_t^2 - c^2 \Delta,
$$
and we make the assumption that $Q^1$ and $Q^c$ vanish quadratically
$$
Q^1(u,v),Q^c(u,v) = O(|u|^2+|v|^2).
$$
The data $u^1_0,u^1_1,u_0^c,u^c_1$ will be chosen small, smooth and localized in space.

This equation models the nonlinear interaction of two types of Klein-Gordon waves, one propagating fast, the other slowly. If $c$ is very large with respect to 1, it is for instance a toy model for the Euler-Maxwell equation describing plasmas; in that case, the fast waves would be electromagnetic, and the slow waves accoustic\footnote{To be slightly more precise, the linearization of the Euler-Maxwell equation around an equilibrium state of constant density, zero velocity, and zero electromagnetic fields is essentially given by a system of Klein-Gordon equations with very different velocities; as for the nonlinearity, it is quadratic and higher order, but not semilinear as in the present paper.}.

Another source of interest of this equation is mathematical: the space-time resonant structure of~(\ref{mandarin}) has features which have not been studied yet, and which are representative of a large class of equations: this is detailed in Section~\ref{bruant}.

\subsection{Known results on global solutions of semilinear wave and Klein-Gordon equations}

We shall quickly review some results on global solutions of wave and Klein-Gordon equations. We shall focus on {\bf the equation set in $\mathbb{R}^{3+1}$}, and on {\bf quadratic nonlinearities $Q= Q(u,\partial u)$}.

\subsubsection{Scalar Klein-Gordon equation}
\label{chardonneret1}
It is important to notice first that in dimension 3 and for the Klein-Gordon equation, 2 is the Strauss exponent. This means that a quadratic nonlinearity (as opposed to any larger power $|u|^{2+\epsilon}$) is barely too weak for dispersive and Strichartz estimates alone to give global existence and scattering. In other words, resonances need to be taken into account.

Global existence in the case of small and localized data for $\square u + u = Q(u,\partial u)$ was proved independently by Shatah~\cite{S} (by the normal form method) and by Klainerman~\cite{K1} (by the vector fields method).

Another line of research is the study of weakly decaying data, which somehow corresponds to ignoring space resonances (see Section~\ref{secspacetime}); we mention in particular Delort and Fang~\cite{DF}, where $X^{s,b}$ methods are employed.

\subsubsection{Scalar wave equation}
\label{chardonneret2}
Still in dimension 3, this problem is much more delicate than Klein-Gordon: indeed, the decay given by the linear wave equation is $\frac{1}{t}$, which makes a quadratic nonlinearity short-range.

It has been observed that the properties of the equation $\square u = Q(u, \partial u)$ depend crucially on the structure of $Q$. If $Q(u, \partial u) = (\partial_t u)^2$ or $u \partial_t u$, John~\cite{J2} was able to prove finite time blow up; but for $Q$ being the null form $(\partial_t u)^2 - (\nabla u)^2$, Klainerman~\cite{K2} and Christodoulou~\cite{C} obtained global solutions for small and localized data. Klainerman relied on the vector fields method, whereas Christodoulou used a change of independent variables (conformal mapping).

\subsubsection{Systems of wave equations}
\label{chardonneret3}
Consider the system $\square_{c_i} u_i = Q_i(\partial u)$, where $u = (u_i)$ is a vector, we assume the velocities to be different, $c_i \neq c_j$ if $i \neq j$, and $Q_i$ is a quadratic polynomial, $Q_i = \sum_{\alpha \beta j k} a^{\alpha \beta jk}_i \partial_\alpha u_i \partial_\beta u_j$. 

Yokoyama~\cite{Y} (and Sideris and Tu in the quasilinear case~\cite{ST}) proved that small and localized data yield a global solution if the interaction of a wave with itself has a null form structure; or more precisely if for each $i$ the quadratic form $\sum_{\alpha \beta} a^{\alpha \beta ii}_i \partial_\alpha \cdot \partial_\beta \cdot$ is the null form of Christodoulou and Klainerman. His proof relies on the vector fields method. Notice that, in particular, the interaction between waves $u_i$, $u_j$, with different speeds might be any polynomial depending only on $\partial u$.

An example of Ohta~\cite{O} shows that, if $Q$ is allowed to depend on $u$ (and not only its derivatives), then interactions between different waves might produce blow up in finite time. For more recent developments, see Katayama and Yokoyama~\cite{KY}.

\subsubsection{Systems of Klein-Gordon equations}
\label{chardonneret4}

\underline{Different masses.}
Systems of Klein-Gordon equations which have been considered have the form $\square u_i + m_i^2 u_i = Q_i(u,\partial u)$, where $u = (u_i)$ is a vector, the masses $m_i$ are assumed to be different, and $Q_i$ are quadratic polynomials.

In space dimension 3, Tsutsumi~\cite{T} studied a slightly different system, by considering the Maxwell Higgs equation in a neighbourhood of an equilibrium. He proved global existence for small and localized data; see also Hayashi, Naumkin and Wibowo~\cite{HNW}.

We also mention the case of dimension 2, since the resonance condition appears more clearly. Tsutsumi proved that, if for any $i,j$, $m_i \neq 2 m_j$, then a global solution exists for small and localized data. Finally, Delort, Fang, and Xue~\cite{DFX} were able to show that if for some $i,j$, $m_i = 2 m_j$, then the global existence result still holds if the nonlinearity has a null form structure.

\bigskip
\underline{Different velocities.} The equation~(\ref{mandarin}), which is the subject of the present article, is a system of Klein-Gordon equations with different velocities. It might be surprising that this case has remained unsolved, whereas, for instance, systems of wave equations (whose decay is much lower than for Klein-Gordon) with different velocities could be treated. We will argue in the following subsection that the crucial concept is that of space-time resonances; and that equation~(\ref{mandarin}) has new features in this respect.

\subsection{Space-time resonances}

\label{secspacetime}

We first discuss briefly the concept of space-time resonance, introduced in Germain, Masmoudi and Shatah~\cite{GMS1}~\cite{GMS2}~\cite{GMS3}; we later explain how it sheds light on the already mentioned results.

\subsubsection{General presentation}

\underline{Transformation of the equation}
In order to present the idea of space-time resonance, consider a fairly general scalar quadratic nonlinear dispersive equation
$$
\left\{ \begin{array}{l} 
i \partial_t u + P(D) u = Q(u,\bar u) \\ u(t=0) = u_0, \end{array} \right.
$$
where $P(D)$ is a real Fourier multiplier, and $Q(u,\bar u)$ is either $u^2$, or $\bar u ^2$ or $u \bar u$. Switching to the unknown function $f = e^{-itP(D)} u$ (the ``profile''), and to the Fourier side, Duhamel's formula can be written
\begin{equation}
\label{pic}
\widehat{f}(t,\xi) = \widehat{u_0}(\xi) + \frac{1}{(2\pi)^{d/2}} \int_0^t \int e^{i s\phi(\xi,\eta)} \widehat{f}(s,\eta)  \widehat{f}(s,\xi-\eta) \,d\eta \,ds,
\end{equation}
(where we denoted for simplicity indifferently $\widehat{f}$ for $\widehat{f}$ or $\widehat{\bar f}$) with
$$
\phi(\xi,\eta) = P(\xi) \pm P(\eta) \pm P(\xi-\eta),
$$
where the signs $--$, $++$ and $+-$ correspond respectively to $Q$ being $u^2$, $\bar u^2$, and $u \bar u$. 

\bigskip
\underline{The resonant sets}
Viewing the integral in~(\ref{pic}) from the point of view of the stationary phase lemma, the critical sets are those where $s\phi$ is not oscillating in $s$, $\eta$, or even worse, both: 
\begin{equation*}
\begin{split}
& \mathcal{T} \overset{def}{=} \{ (\xi,\eta) \;\mbox{such that}\; \phi(\xi,\eta) = 0 \} \;\;\;\;\mbox{(``time resonances'')} \\
& \mathcal{S} \overset{def}{=} \{ (\xi,\eta) \;\mbox{such that}\;\partial_\eta \phi(\xi,\eta) = 0 \} \;\;\;\;\mbox{(``space resonances'')}\\
& \mathcal{R} \overset{def}{=} \mathcal{S} \cap \mathcal{T} \;\;\;\;\mbox{(``space-time resonances'')}.
\end{split}
\end{equation*}
The central idea is that the sets $\mathcal{T}$, $\mathcal{S}$, and to a greater extent $\mathcal{R}$, are the obstructions to a linear behaviour of $u$, for large time, and small data.

\bigskip

Of course, similar manipulations can be performed for any nonlinear dispersive equation (see below for an example); this leads to a formula similar to~(\ref{pic}), and once $\phi$ is known, the sets $\mathcal{T}$, $\mathcal{S}$ and $\mathcal{R}$ can be defined as above.

Finally, the space and time resonant sets have a physical interpretation: $\mathcal{T}$ corresponds to bilinear interactions of plane waves which are resonant in the classical sense; $\mathcal{S}$ corresponds to bilinear interactions of wave packets whose group velocities are equal.

\bigskip \underline{The method} The method which we apply is straightforward: perform a (time-dependent) cut-off in the $(\xi,\eta)$ space in order to distinguish three regions. Away from $\mathcal{T}$, integrate by parts in $s$ (which amounts to a normal form transform). Away from $\mathcal{S}$, integrate by parts in $\eta$ (this is similar to the vector fields method). There remains a neighbourhood of $\mathcal{R}$; it should shrink with $t$, and one has to take advantage of the smallness of this set.

\subsubsection{An interpretation of the already mentioned results}

The wave as well as the Klein-Gordon equation can be described as a sum of a $+$ and a $-$ wave; the dispersion relations are
\begin{equation*}
\begin{split}
& P(\xi) = \pm c |\xi|\;\;\;\;\mbox{(linear wave equation with velocity $c$)} \\
& P(\xi) = \pm \sqrt{m^2 + c^2 |\xi|^2} \;\;\;\;\mbox{(linear Klein-Gordon equation with mass $m$ and velocity $c$)}
\end{split}
\end{equation*}
Then one can define $\phi$ as above for each of the possible bilinear interactions. For instance, in the case of the system $\square_{c_i} u_i = Q_i(\partial u)$, one must consider interactions between $+$ or $-$ linear waves travelling at different velocities. It leads to
$$
\phi^{\pm, \pm,\pm}_{ijk} = \pm c_i |\xi| \pm c_j |\eta| \pm c_k |\xi-\eta|;
$$
resonant sets are then defined in an obvious way.

The space-time resonances approach that was just sketched gives a possible interpretation of the results mentioned above.
\begin{itemize}
\item In the case of the scalar Klein-Gordon equation (Subsection~\ref{chardonneret1}), $\mathcal{R}=\emptyset$; that is why the normal form method can be applied globally.
\item For the scalar wave equation (Subsection~\ref{chardonneret2}), $\mathcal{R}$ is very large, which is an obstruction to global existence and scattering. This explains the role of the null form structures: it is precisely for these nonlinearities that the interaction term vanishes on $\mathcal{R}$.
\item Systems of wave equation with different speeds (Subsection~(\ref{chardonneret3}) only have space-time resonances in the following configuration: two linear waves corresponding to a given velocity $c_i$ interact to give a third linear wave also at velocity $c_i$. Thus a null form is only needed for this kind of interactions.
\item Finally, for Klein-Gordon equations with different masses (Subsection~(\ref{chardonneret4})), space-time resonances may occur only in case $m_i = 2m_j$. For any of these resonant interactions, one observes that $\partial_\xi \phi$ vanishes on $\mathcal{R}$, which helps in controlling the nonlinearity: see~\cite{GMS2} and~\cite{GMS3} for instances of this phenomenon. It might explain why global existence can be obtained in dimension 3 even in the presence of these resonances.
\end{itemize}

\subsubsection{Application to our problem}

For the problem which is the subject of this article (equation~(\ref{mandarin})), one needs to define several phase functions corresponding to all the possible interactions. They read
$$
\phi^{k,\ell, m}_{\epsilon_0,\epsilon_1,\epsilon_2}(\xi,\eta) \overset{def}{=} \epsilon_0 \< \xi \>_k - \epsilon_1 \< \eta \>_\ell - \epsilon_2 \<\xi-\eta\>_m,
$$
where $k,l,m$ equal $1$ or $c$ and $\epsilon_0,\epsilon_1,\epsilon_2$ equal $+$ or $-$ (see Section~\ref{secfirstrec} for the full derivation). The associated time, space, and space-time resonant sets are
\begin{equation}
\begin{split}
& \mathcal{T}_{\epsilon_0,\epsilon_1,\epsilon_2}^{k,\ell,m} \overset{def}{=} \{ (\xi,\eta) \;\mbox{such that} \;\phi^{k,\ell, m}_{\epsilon_0,\epsilon_1,\epsilon_2} = 0 \} \\
& \mathcal{S}_{\epsilon_0,\epsilon_1,\epsilon_2}^{k,\ell,m} \overset{def}{=} \{ (\xi,\eta) \;\mbox{such that} \;\partial_\eta \phi^{k,\ell, m}_{\epsilon_0,\epsilon_1,\epsilon_2} = 0 \} \\
& \mathcal{R}_{\epsilon_0,\epsilon_1,\epsilon_2}^{k,\ell,m} \overset{def}{=} \mathcal{T}_{\epsilon_0,\epsilon_1,\epsilon_2}^{k,\ell,m} \cap \mathcal{S}_{\epsilon_0,\epsilon_1,\epsilon_2}^{k,\ell,m}.
\end{split}
\end{equation}

We will see in Section~\ref{secquadres} that space time resonances occur for some interactions. Then $\mathcal{R}_{\epsilon_0,\epsilon_1,\epsilon_2}^{k,\ell,m}$ has dimension 2 and is of the form
$\{ |\eta|=R\;,\;\xi = \lambda \eta\}$ for real numbers $R$ and $\lambda$. Furthermore (in general), $\partial_\xi \phi$ does not vanish on $\mathcal{R}$.

Thus, known methods of proof do not seem to apply here; a new approach is needed, which will be explained in the next section.

\subsection{Notations}

We use the following notations: 
\begin{itemize}
\item Japanese brackets: $\<x\> \overset{def}{=} \sqrt{1 + x^2} \;\;\;\mbox{and}\;\;\; \<x\>_c \overset{def}{=} \sqrt{1 + c^2 x^2}$.
\item Fourier transform: $[ \mathcal{F} f](\xi)  = \widehat{f}(\xi) \overset{def}{=} \frac{1}{(2\pi)^{d/2}} \int e^{-ix\xi} f(x) \,dx$.
\item Fourier multiplier: $\mathcal{F}(m(D)f)(\xi) \overset{def}{=} m(\xi) \widehat{f}(\xi)$.
\item Pseudo-product operator with symbol $m$: $T_m(f,g) \overset{def}{=} \mathcal{F}^{-1} \int m(\xi,\eta) \widehat{f}(\eta) \widehat{f}(\xi-\eta) \,d\eta$.
\item Inhomogeneous Sobolev spaces: $\|u\|_{H^s} \overset{def}{=} \|\<D\>^s u \|_2 \;\;\;\;\mbox{and}\;\;\;\;\|u\|_{W^{s,p}} \overset{def}{=} \|\<D\>^s u \|_p$.
\item Space time norms: for instance $L^p W^{s,q}$ stands for $L^p_t([0,\infty),W^{s,q}_x)$.
\item $\delta$-neighbourhood of $E$: $B_\delta(E) \overset{def}{=} \left\{ x \;\mbox{such that} \operatorname{dist}(x,E) \leq \delta \right\}$.
\item Inequalities with implicit constants: $A \lesssim B$ (respectively $A \gtrsim B$) indicates that, for a constant $C$ (depending on the context), $A \leq C B$ (respectively $A \geq CB$). Finally $A\sim B$ if $A \lesssim B$ and $B\lesssim A$.
\end{itemize}

\section{Main result and ideas of the proof}

\subsection{Outcome, source, and separation of resonances}

We introduce new concepts which will be crucial when dealing with the system~(\ref{mandarin}).

\begin{df} A frequency $X$ is called the \textit{outcome} of a space-time resonance if for some $\eta$, $(X,\eta)$ belongs to $\cup_{k,\ell,m,\epsilon_0,\epsilon_1,\epsilon_2} \mathcal{R}_{\epsilon_0,\epsilon_1,\epsilon_2}^{k,\ell,m}$. Loosely speaking, such a frequency might be created by a space-time resonant interaction.

A frequency $X$ is called the \textit{source} of a space-time resonance if either $(\xi,X)$ for some $\xi$, or $(X+\eta,\eta)$ for some $\eta$, belongs to $\cup_{k,\ell,m,\epsilon_0,\epsilon_1} \mathcal{R}_{\epsilon_0,\epsilon_1,\epsilon_2}^{k,\ell,m}$. Loosely speaking, such a frequency might have a space-time resonant interaction with another one.
\end{df}

Intuitively, space-time resonances can feed themselves if some sources of space-time resonances are also outcome of space-time resonances. If this does not happen, we say that the resonances are separated. 

\begin{df} The \textit{resonances separation condition} holds if no outcome frequency of a space-time resonance is also a source frequency of a space-time resonance.
\end{df}

\subsection{The main theorem}

\begin{thm}
\label{maintheorem}
Assume that the resonances separation condition is satisfied for~(\ref{mandarin}).
Then there exists an integer $N$, and a constant $\epsilon>0$ such that if 
$$
\left\|(u_0^1,u_0^c)\right\|_{H^N} + \left\|(u_1^1,u_1^c)\right\|_{H^{N-1}} + \||x|(u_0^1,u_0^c)\|_2 + \left\||x|\frac{1}{|D|}(u_1^1,u_1^c)\right\|_{2} < \epsilon,
$$
then there exists a global solution to~(\ref{mandarin}) such that
$$
\mbox{for any $t$},\;\;\;\left\|(u^1,u^c) \right\|_{H^N} \lesssim \epsilon\;\;\;\;\mbox{and} \;\;\;\;\left\|(u^1,u^c) \right\|_3 \lesssim \frac{\epsilon}{\sqrt{t}},
$$
and furthermore $u$ scatters in $H^N$: for $k=1,c$, there exist solutions $U^k$ of $\square U^k + U^k = 0$ with data in $H^N \times H^{N-1}$ such that
$$
\left\| u^1 - U^1 \right\|_{H^N} + \left\| u^c - U^c \right\|_{H^{N}} \longrightarrow 0\;\;\;\;\mbox{as $t \rightarrow \infty$}.
$$
\end{thm}
When are resonances of~(\ref{mandarin}) separated?
We believe that it is always the case, except maybe for exceptional values of $c$, but this is not so easy to prove in general: one has to compare solutions of rather complicated algebraic equations.
However, the resonances separation condition can be very easily checked numerically for given values of $c$. In the Appendix, we present elementary computations which show that resonances are separated for $c=5$.

\subsection{Genericity of the problem}

\label{bruant}

First notice that a modification of the proof of Theorem~\ref{maintheorem} would give global solutions and scattering for nonlinear dispersive equations such that some general assumptions and furthermore: the decay of the linear part is the same as for Klein-Gordon; the nonlinearity is quadratic; and resonances are separated.

\bigskip

From the point of view of space-time resonances, the problem under study has features which seem generic, and which had not been examined before:
\begin{itemize}
\item The dimension of the space-time resonant set $\mathcal{R}$ is 2 in the 6-dimensional frequency space $(\xi,\eta)$. This is natural since, under nondegeneracy assumptions, $\mathcal{T}$ has dimension 5, and $\mathcal{S}$ has dimension 3.
\item The space-time resonant set $\mathcal{R}$ is not simply a linear subspace, as in the previous works~\cite{GMS1}~\cite{GMS2}~\cite{GMS3}. This can be traced back to the dispersion relation being nonhomogeneous.
\item The particular form of the space time resonant sets ($\{ |\xi| = R \;,\;\eta = \lambda \xi \}$ for real numbers $R$, $\lambda$) is to be expected (in general) if the dispersion relations $P(\xi)$ only depend on $|\xi|$.
\item The $\xi$ derivative $\partial_\xi \phi$ does not (generically in $c$) vanish on $\mathcal{R}$: this of course should be expected in general.
\end{itemize}

\subsection{Ideas of the proof}

Local existence is standard; thus the proof will consist in proving global a priori estimates, which will be stated in Section~\ref{gypaete}. They will be derived by splitting the $(\xi,\eta)$ space.

\subsubsection{Splitting the $(\xi,\eta)$ space.}
\label{secsplit}
Let us first write Duhamel's formula for one of the possible quadratic interactions occuring in~(\ref{mandarin}) (we drop all indices and inessential constants, see Section~\ref{secfirstrec} for the exact expression)
\begin{equation*}
\widehat{f}(\xi) = \widehat{f}(\xi) + \int_0^t \int e^{-i s \phi} \frac{\widehat{f} (\eta)}{\<\eta\>} \frac{ \widehat{f} (\xi-\eta)}{\< \xi-\eta \>} \, d\eta\,ds.
\end{equation*}
We split the above as follows
$$
\widehat{f}(\xi) = \widehat{f_0}(\xi) - \int_0^t \int e^{-is\phi} \left[\chi_\mathcal{T}^s + \chi_\mathcal{S}^s + \chi_\mathcal{R}^s \right] \frac{\widehat{f} (\eta)}{\<\eta\>} \frac{ \widehat{f} (\xi-\eta)}{\< \xi-\eta \>} \, d\eta\,ds
$$
where $\chi_\mathcal{T}^t$, $\chi_\mathcal{S}^t$, and $\chi_\mathcal{R}^t$ are time-dependent cut-off functions localizing respectively away from $\mathcal{S}$, away from $\mathcal{T}$, and in a neighbourhood of $\mathcal{R}$ which shrinks as $t \rightarrow \infty$. As already explained, the idea in order to obtain the desired estimates is to integrate by parts in $\eta$ the term containing $\chi_\mathcal{T}^t$; in $s$ the term containing $\chi_\mathcal{S}$; and to use the (increasing) smallness of the support of $\chi_\mathcal{R}^t$ in order to estimate the last term.

We describe next a few difficulties which one must face when applying this plan; and how they are overcome.

\subsubsection{Singularities of the cut-off functions, and the associated pseudo-product operators.} As $t$ converges to infinity, the cut-off functions become singular along the set $\mathcal{R}$, which is a surface with a non-zero curvature. One is thus led to estimating pseudo-product operators with singularities along surfaces with non-trivial geometry. Achieving a deep understanding of these operators seems quite difficult, but rough bounds can be obtained easily. Why is that enough? The idea is that the nonlinearity is exactly at the Strauss exponent, which means that only a very small gain (in time decay) is needed in order to close the estimates. Actually, even our very crude bounds suffice to give this gain.

\subsubsection{The high frequencies.} When performing the manipulations explained in Section~\ref{secsplit}, pseudo-products appear which are singular not only close to $\mathcal{R}$, but also at infinity, in the sense that they are not asymptotically homogeneous of degree 0 there. The idea is to treat the high frequencies by an argument independent of resonances (essentially, Strichartz estimates), which gives that $u$ can be controlled in $H^N$ for $N$ very large. We now split the quadratic term into frequencies less, or more, than $t^\delta$, where $\delta$ is very small. The piece containing frequencies less than $t^\delta$ can be controlled, for the singularity at $\infty$ of the pseudo-products is barely felt; and the piece containing frequencies more than $t^\delta$ can be estimated directly, if $N\delta$ is big enough.

\subsubsection{The separation of resonances condition.} It is used in the following way: let $\chi_\mathcal{O}$ localize smoothly to a (very small) neighbourhood of the frequencies which are the outcome of a space-time resonance, and let $\widetilde{\chi}_\mathcal{O}$ satisfy $\chi_\mathcal{O} + \widetilde{\chi}_\mathcal{O} = 1$. Then $\widetilde{\chi}_\mathcal{O}(D) u$ will satisfy stronger estimates than $\chi_\mathcal{O} (D)u$.

In order to prove the estimates on $\widetilde{\chi}_\mathcal{O}(D) u$, we use the fact 
$\widetilde{\chi}_\mathcal{O}(D) u$ does not see space-time resonances; or in other words $\widetilde{\chi}_\mathcal{O}(\xi) \chi_\mathcal{R}^t(\xi,\eta) = 0$, as follows immediately from the definitions.

In order to prove the estimates on $\chi_\mathcal{O}(D) u$, we use the separation of resonances condition. It implies that, as far as nearly space-time resonant frequencies are concerned, $\chi_\mathcal{O}(D) u$ can be written as a quadratic expression in $\widetilde{\chi}_\mathcal{O}(D) u$. This is helpful since the estimates on $\widetilde{\chi}_\mathcal{O}(D) u$ are stronger.

\section{Examination of the resonances}

\subsection{Duhamel's formula in Fourier space}

\label{secfirstrec}

It will first of all be convenient to assume that $Q^1$ and $Q^c$ are quadratic polynomials; treating terms of order three and higher is very easy, so we can forget about them. Thus in the following
$$
Q^1(u^1,u^c) = \alpha (u^1)^2 + \beta (u^c)^2 + \gamma u^1 u^c \;\;\;\;\;\;Q^c(u^1,u^c) = \delta (u^1)^2 + \epsilon (u^c)^2 + \zeta u^1 u^c.
$$
for some real constants $\alpha,\beta,\gamma,\delta,\epsilon,\zeta$.
Next we diagonalize the linear part of the equation by adopting the new unknown functions
\begin{equation*}
\begin{split}
u^k_{\pm} \overset{def}{=} \partial_t u^k \pm i \< D \>_k u^k \;\;\;\;\;\;\mbox{for $k=1$ or $c$}\\
\end{split}
\end{equation*}
which are associated to the profiles
$$
f^k_{\pm} \overset{def}{=} e^{\mp i t \<D\>_k} u_{\pm}
$$
and the initial data
$$
u^k_{\pm,0} = f^k_{\pm,0} \overset{def}{=} u^k_1 \pm i \< D \>_k u^k_0.
$$
Writing Duhamel's formula for u gives for $k=1,c$ and $\epsilon_0 = \pm$,
\begin{equation}
\label{pingouin}
\begin{split}
u^k_{\epsilon_0} = e^{i\epsilon_0 t \<D\>} u^k_{\epsilon_0,0} + \sum_{\epsilon_1,\epsilon_2 = \pm} \;\sum_{\ell, m = 1,c} a_{\epsilon_0, \epsilon_1, \epsilon_2}^{k,\ell,m} \int_0^t \int e^{i \epsilon_0 (t-s)\<D\>} \frac{u^\ell_{\epsilon_1}}{\<D\>_\ell} \frac{u^m_{\epsilon_2}}{\<D\>_m}\,ds
\end{split}
\end{equation}
where the $a_{\epsilon_0,\epsilon_1, \epsilon_2}^{k,\ell,m}$ are real coefficients (which can be expressed in terms of $\alpha,\beta,\gamma,\delta,\epsilon,\zeta$). Equivalently, the equations for the profiles in Fourier space are
\begin{equation*}
\begin{split}
\widehat{f^k_{\epsilon_0}}(\xi) = \widehat{f_{\epsilon_0,0}^k}(\xi) + \sum_{\epsilon_1,\epsilon_2 = \pm} \; \sum_{\ell, m = 1,c} \frac{a_{\epsilon_0, \epsilon_1, \epsilon_2}^{k,\ell,m}}{(2\pi)^{d/2}} \int_0^t \int e^{-i s \phi^{k, \ell, m}_{\epsilon_0, \epsilon_1, \epsilon_2}} \frac{\widehat{f^\ell_{\epsilon_1}} (\eta)}{\<\eta\>_\ell} \frac{ \widehat{f^m_{\epsilon_2}} (\xi-\eta)}{\< \xi-\eta \>_m} \, d\eta\,ds
\end{split}
\end{equation*}
where the phases are given by
$$
\phi^{k,\ell, m}_{\epsilon_0,\epsilon_1,\epsilon_2}(\xi,\eta) \overset{def}{=} \epsilon_0 \< \xi \>_k - \epsilon_1 \< \eta \>_\ell - \epsilon_2 \<\xi-\eta\>_m.
$$

\subsection{Quadratic resonances}

\label{secquadres}

Following the space-time resonance method explained in the introduction, we need to compute the sets
\begin{equation}
\begin{split}
& \mathcal{T}_{\epsilon_0,\epsilon_1,\epsilon_2}^{k,\ell,m} \overset{def}{=} \{ \phi^{k,\ell, m}_{\epsilon_0,\epsilon_1,\epsilon_2} = 0 \}\;\;\;\mbox{(time resonances)} \\
& \mathcal{S}_{\epsilon_0,\epsilon_1,\epsilon_2}^{k,\ell,m} \overset{def}{=} \{ \partial_\eta \phi^{k,\ell, m}_{\epsilon_0,\epsilon_1,\epsilon_2} = 0 \}\;\;\;\mbox{(space resonances)}
\end{split}
\end{equation}
and their intersection
$$
\mathcal{R}_{\epsilon_0,\epsilon_1,\epsilon_2}^{k,\ell,m} \overset{def}{=} \mathcal{T}_{\epsilon_0,\epsilon_1,\epsilon_2}^{k,\ell,m} \cap \mathcal{S}_{\epsilon_0,\epsilon_1,\epsilon_2}^{k,\ell,m} \;\;\;\mbox{(space-time resonances)}.
$$
The analysis of these sets requires some computations, which we delegate to the Appendix. We will only use the results given below, lemmas~\ref{aigle1} and~\ref{aigle2}.

\begin{lem}
\label{aigle1}
(i) The space time resonant sets are either empty, or of the form
$$
\mathcal{R}_{\epsilon_0,\epsilon_1,\epsilon_2}^{k,\ell,m} = \cup_{j=1}^{J_{\epsilon_0,\epsilon_1,\epsilon_2}^{k,\ell,m}} \{ |\eta| = R_{\epsilon_0,\epsilon_1,\epsilon_2,j}^{k,\ell,m} \;,\;\xi = \lambda_{\epsilon_0,\epsilon_1,\epsilon_2,j}^{k,\ell,m} \xi \}\,\,,
$$
where $J_{\epsilon_0,\epsilon_1,\epsilon_2}^{k,\ell,m}$ is an integer, and $R_{\epsilon_0,\epsilon_1,\epsilon_2,j}^{k,\ell,m}$ and $\lambda_{\epsilon_0,\epsilon_1,\epsilon_2,j}^{k,\ell,m}$ are real numbers.
\smallskip

(ii) The intersection of $\mathcal{S}_{\epsilon_0,\epsilon_1,\epsilon_2}^{k,\ell,m}$ and $\mathcal{T}_{\epsilon_0,\epsilon_1,\epsilon_2}^{k,\ell,m}$ is of finite order. This means that there exists an integer $n$ such that (dropping indices for clarity)
$$
\mbox{as $\epsilon \rightarrow 0$,}\;\;\;\operatorname{dist}([\partial B_\epsilon(\mathcal{R})]\cap \mathcal{S}\,,\,[ \partial B_\epsilon(\mathcal{R}) ] \cap \mathcal{T}) \gtrsim \epsilon^n.
$$
\end{lem}

Thus space-time resonances are separated if, for any set of indices,
$$
\lambda_{\epsilon_0,\epsilon_1,\epsilon_2,j}^{k,\ell,m} R_{\epsilon_0,\epsilon_1,\epsilon_2,j}^{k,\ell,m} \neq R_{\epsilon_0',\epsilon_1',\epsilon_2',j'}^{k',\ell',m'}\;\;\;\;\mbox{and}\;\;\;\;\left| \lambda_{\epsilon_0,\epsilon_1,\epsilon_2,j}^{k,\ell,m} - 1 \right| R_{\epsilon_0,\epsilon_1,\epsilon_2,j}^{k,\ell,m} \neq \lambda_{\epsilon_0',\epsilon_1',\epsilon_2',j'}^{k',\ell',m'} R_{\epsilon_0',\epsilon_1',\epsilon_2',j'}^{k',\ell',m'}.
$$

\subsection{The cut-off functions}

\label{subsecdef}

\subsubsection{Definition of $\theta$}

First pick $M$ such that all the space time resonant sets are contained in the ball of radius $M/2$: $\cup R_{\epsilon_0,\epsilon_1,\epsilon_2,j}^{k,\ell,m} \subset B(0,M/2)$.

It will be necessary in the proof to distinguish between high and low frequencies. To this end, we introduce the cut off function $\theta(\xi,\eta)$, which is such that
\begin{equation}
\label{colvert}
\theta \in \mathcal{C}^\infty_0\;\;\;\;,\;\;\;\;\theta = 1 \;\mbox{on $B(0,M)$}\;\;\;\;\mbox{and}\;\;\;\;\theta = 0 \;\mbox{on $B(0,M+1)^c$}.
\end{equation}

\subsubsection{Definition of $\chi_\mathcal{O}$, $\widetilde{\chi}_{\mathcal{O}}$}

Define the union of all frequencies which are the outcome of a space-time resonance
$$
\mathcal{O} \overset{def}{=} \cup_{\epsilon_0,\epsilon_1,\epsilon_2 , k,\ell,m,j} \left\{ \xi \;\;\mbox{such that there exists $\eta$ with $(\xi,\eta) \in \mathcal{R}_{\epsilon_0,\epsilon_1,\epsilon_2}^{k,\ell,m,j}$} \right\}.
$$
Since space-time resonances are separated, it is possible to find $\delta_0$ such that no frequency in $B_{10\delta_0}(\mathcal{O})$ (a $10\delta_0$-neighbourhood of $\mathcal{O}$) is a source of a space-time resonance. Define $\chi_\mathcal{O}$ a smooth cut-off function such that 
\begin{equation*}
\begin{split}
& \chi_\mathcal{O} = 1 \;\;\;\mbox{on $B_{\delta_0/2}(\mathcal{O})$} \\
& \chi_\mathcal{O} = 0 \;\;\;\mbox{outside of $B_{\delta_0}(\mathcal{O})$}
\end{split}
\end{equation*}
and then let $\widetilde{\chi}_{\mathcal{O}}$ satisfy
$$
\chi_\mathcal{O} + \widetilde{\chi}_{\mathcal{O}} = 1.
$$

\subsubsection{Definition of $\chi_\mathcal{R}^\rho$, $\chi_\mathcal{S}^\rho$ and $\chi_\mathcal{T}^\rho$}

The cut-off functions which we are now about to define require some care: they cut off to a distance $\rho$ of $\mathcal{R}$, and their smoothness depends on $\rho$.

\begin{lem}
\label{aigle2}
For each set of indices $\epsilon_0,\epsilon_1,\epsilon_2, k,\ell,m,j$, it is possible to find cut-off functions
$$
\chi_{\mathcal{T}_{\epsilon_0,\epsilon_1,\epsilon_2,j}^{k,\ell,m}}^\rho(\xi,\eta)\;\;,\;\;\chi_{\mathcal{S}_{\epsilon_0,\epsilon_1,\epsilon_2,j}^{k,\ell,m}}^\rho(\xi,\eta)\;\;,\;\;\chi_{\mathcal{R}_{\epsilon_0,\epsilon_1,\epsilon_2,j}^{k,\ell,m}}^\rho(\xi,\eta).
$$
such that (in the following, we drop the indices $\epsilon_0,\epsilon_1,\epsilon_2 , k,\ell,m,j$ for simplicity)
\begin{itemize}
\item $\chi_\mathcal{R}^\rho$, $\chi_\mathcal{S}^\rho$ and $\chi_\mathcal{T}^\rho$ are smooth.
\item Their sum equals one: $\chi_{\mathcal{T}}^\rho + \chi_{\mathcal{S}}^\rho + \chi_{\mathcal{R}}^\rho = 1$.
\item Outside of $B_{2\delta_0}(\mathcal{R})$, $\chi_{\mathcal{S}}^\rho$ and $\chi_{\mathcal{T}}^\rho$ are independent of $\rho$.
\item The cut-off function $\chi_\mathcal{R}^\rho$ has the following form:
$$
\chi_{\mathcal{R}}^\rho = \sum_j \chi \left( \frac{|\eta|-R_j}{\rho} \right) \chi \left( \frac{\xi - \lambda_j \eta}{\rho} \right),
$$
where $\chi$ is a smooth, compactly supported function. 
\item The derivatives of $\frac{\chi_{\mathcal{S}}^\rho}{\phi}$ and $\frac{\chi_{\mathcal{T}}^\rho \partial_\eta \phi}{|\partial_\eta \phi|^2}$ satisfy if $|(\xi,\eta)| \leq M$:
\begin{equation}
\label{rossignol1}
\mbox{If $|\alpha| \leq 20$,} \;\;\;\;
\left| \partial_{\xi,\eta}^{\alpha} \frac{\chi_{\mathcal{S}}^\rho}{\phi} \right| \;,\; \left| \partial_{\xi,\eta}^{\alpha} \frac{\chi_{\mathcal{T}}^\rho \partial_\eta \phi}{|\partial_\eta \phi|^2} \right| \lesssim \frac{1}{\left[ \rho + \operatorname{dist} \left( (\xi,\eta) , \mathcal{R} \right) \right]^{n}}
\end{equation}
for an integer $n$.
\item And for $|(\xi,\eta)| \geq M$,
\begin{equation}
\label{rossignol2}
\mbox{If $|\alpha| \leq 20$,} \;\;\;\;
\left| \partial_{\xi,\eta}^{\alpha} \frac{\chi_{\mathcal{S}}^\rho}{\phi} \right| \;,\;
\left| \partial_{\xi,\eta}^{\alpha} \frac{\chi_{\mathcal{T}}^\rho \partial_\eta \phi}{|\partial_\eta \phi|^2} \right| \lesssim |\xi,\eta|^{n}
\end{equation}
for an integer $n$.
\end{itemize}
\end{lem}

\subsection{Bounds on the associated pseudo-product operators}

The manipulations which we will perform lead to pseudo product operators with symbols $m$ of the form $\frac{\chi_{\mathcal{S}}^\rho}{\phi}$ or $\frac{\chi_{\mathcal{T}}^\rho \partial_\eta \phi}{|\partial_\eta \phi|^2}$. Such operators are a priori not uniformly bounded (in $\rho$) on Lebesgue spaces satisfying the H\"older relation; as explained above, this is due to singularities at infinity, and along $\mathcal{R}$. The following proposition gives rough, but sufficient for our purposes, bound.

\begin{prop} 
\label{toucan}
(i) If $p$, $q$ and $r$ satisfy $\frac{1}{r} = \frac{1}{p} + \frac{1}{q}$, then uniformly in $\rho$
$$
\left\| T_{ \chi \left( \frac{\xi - \lambda \eta}{\rho} \right)} (f,g) \right\|_r \lesssim \|f\|_p \|g\|_q.
$$
(recall that $\chi$ was defined in the statement of Lemma~\ref{aigle2}).
\medskip

(ii) If $0 \leq s < \frac{3}{2}$,
$$
\left\| \chi \left( \frac{|D|-R}{\rho} \right) f \right\|_2 \lesssim \rho^{s/3} \||x|^s f \|_2.
$$
\medskip

(iii) There exists a constant $A$ such that if
$$ m(\xi,\eta) = \theta(\xi,\eta) \frac{\chi_{\mathcal{S}}^\rho}{\phi} \;\;\;\;\mbox{or}\;\;\;\;\theta(\xi,\eta) \frac{\chi_{\mathcal{T}}^\rho \partial_\eta \phi}{|\partial_\eta \phi|^2},
$$
then
$$
\left\|T_m (f,g) \right\|_r \lesssim \frac{1}{\rho^A} \|f\|_p \|g\|_q \;\;\;\;\mbox{if $\frac{1}{p} + \frac{1}{q} = \frac{1}{r}$}.
$$

\medskip
(iv) There exists a constant, which we also denote $A$, such that if $T\geq 1$ and
$$
m(\xi,\eta) = \left[ \theta\left(\frac{(\xi,\eta)}{T} \right) - \theta(\xi,\eta) \right] \chi_{\mathcal{T}}^\rho \;\;\;\;\mbox{or}\;\;\;\; \left[ \theta\left(\frac{(\xi,\eta)}{T} \right) - \theta(\xi,\eta) \right] \chi_{\mathcal{S}}^\rho,
$$
then
$$
\left\|T_m (f,g) \right\|_r \lesssim T^A \|f\|_p \|g\|_q\;\;\;\;\mbox{if $\frac{1}{p} + \frac{1}{q} = \frac{1}{r}$}.
$$

\medskip
(v) There exists a constant, which we still denote $A$, such that if $T \geq 1$ and
$$
m(\xi,\eta) = \widetilde{\chi}_\mathcal{O}(\xi) \chi_{\mathcal{S}}(\xi,\eta) \theta\left(\frac{(\xi,\eta)}{T}\right)\frac{1}{\phi}\;\;\;\;\mbox{or}\;\;\;\;\widetilde{\chi}_\mathcal{O}(\xi) \chi_{\mathcal{T}}(\xi,\eta) \theta\left(\frac{(\xi,\eta)}{T}\right)\frac{\partial_\eta \phi}{|\partial_\eta \phi|^2},
$$
then, for $0\leq s\leq 1$ and $\frac{1}{p} + \frac{1}{q} = \frac{1}{r}$,
$$
\left\|| x|^s T_m (f,g) \right\|_r \lesssim T^A \|\<x\>^s f\|_p \|g\|_q\;\;\;\;\mbox{and}\;\;\;\;\left\|| x|^s T_m (f,g) \right\|_r \lesssim T^A \| f\|_p \|\<x\>^s g\|_q.
$$
\end{prop}

\begin{proof} (i) follows from Proposition~\ref{geai}; (iii) and (iv) from Lemma~\ref{aigle2} and Corollary~\ref{loriot}; (v) results from the interpolation between the cases $s=0$ and $s=1$, which are simple.
Finally, (ii) is a particular case of the general inequality
$$
\left\|m(D) f \right\|_2 \lesssim \|m\|_{3/s} \left\||x|^s f\right\|_2,
$$
which is a consequence of the Plancherel and Sobolev embedding theorems:
$$
\left\|m(D) f \right\|_2 = \left\| m(\xi) \widehat{f}(\xi) \right\|_2 \leq \left\|m\right\|_{3/s} \|\widehat{f}\|_{\left(\frac{1}{2} - \frac{s}{3} \right)^{-1} } \lesssim \left\|m\right\|_{3/s} \||D|^s \widehat{f}\|_2 = \left\|m\right\|_{3/s} \left\||x|^s f\right\|_2. 
$$
\end{proof}

\section{Proof of the main theorem}

\subsection{The a priori estimate}

\label{gypaete}

The proof of the theorem will essentially consist in the following global a priori estimate ($\delta_1$ and $N$ are constants whose precise values will be fixed in the following):
\begin{subequations}
\begin{align}
& \label{canard1} \left\| u^k_{\pm} \right\|_{H^N} \lesssim 1 \;\;\;\;\mbox{(regularity in $L^2$)} \\
& \label{canard2} \left\| u^k_{\pm} \right\|_{L^{\left(\frac{1}{3}-\delta_1\right)^{-1}}} \lesssim \frac{1}{\<t\>^{\frac{1}{2}+3\delta_1}} \;\;\;\;\mbox{(decay slighlty above $L^3$)} \\
&\label{canard3} \left\| \widetilde{\chi}_{\mathcal{O}}(D) u^k_{\pm} \right\|_{L^{\left(\frac{1}{6}+\delta_1\right)^{-1}}} \lesssim \frac{1}{\<t\>^{1-3\delta_1}} \;\;\;\;\mbox{(decay below $L^6$ for ``non-outcome'' frequencies)}\\
&\label{canard4} \left\| |x| f^k_{\pm} \right\|_2 \lesssim \sqrt{\<t\>} \;\;\;\;\mbox{(localization in $L^2$)} \\
&\label{canard5} \left\| |x|^{1/8} \widetilde{\chi}_{\mathcal{O}}(D) f^k_{\pm} \right\|_2 \lesssim 1 \;\;\;\;\mbox{(localization in $L^2$ for ``non-outcome'' frequencies).}
\end{align}
\end{subequations}
Observe that interpolating between~(\ref{canard1}) and~(\ref{canard2}) gives
\begin{equation}
\label{canard6}
\left\|u\right\|_3 \lesssim \frac{1}{\sqrt{\<t\>}}.
\end{equation}
Define the associated norm
\begin{equation*}
\begin{split}
\left\|u\right\|_X & = \sum_{\epsilon=\pm} \sum_{k=1,c} \sup_t \left[ \|u_\epsilon^k\|_{H^N} + \<t\>^{\frac{1}{2}+3\delta_1} \left\| u_\epsilon^k \right\|_{L^{\left(\frac{1}{3}-\delta_1\right)^{-1}}} + \<t\>^{1-3\delta_1} \left\| \widetilde{\chi}_{\mathcal{O}}(D) u_\epsilon^k \right\|_{L^{\left(\frac{1}{6}+\delta_1\right)^{-1}}}\right. \\
& \;\;\;\;\;\;\;\;\;\;\;\;\;\;\;\;\;\;\;\;\;\;\;\;\;\;\;\;\;\;\;\;\;\;\;\;\;\; + \left. \frac{1}{\sqrt{t}} \left\||x| f_\epsilon^k\right\|_2 + \left\| |x|^{1/8} \widetilde{\chi}_{\mathcal{O}} f^k_{\pm} \right\|_2 \right].
\end{split}
\end{equation*}

The a priori estimate which we will prove is that under the assumptions of Theorem~\ref{maintheorem}, $\|u\|_X \lesssim \epsilon$.

\subsection{From the a priori estimate $\|u\|_X \lesssim \epsilon$ to the theorem}

Once we know that $\|u\|_X \lesssim \epsilon$, the proof of the theorem follows in a straightforward way. Indeed, local existence is easily obtained given the available regularity; it is extended to a global result by the a priori estimate. Finally, scattering follows easily, by applying the same kind of estimate as in Subsection~\ref{subsecscat}.

\subsection{Reduction to a quadratic estimate}

We start from Duhamel's formula~(\ref{pingouin}). When manipulating this expression, we will not distinguish between all the possible phases, but simply use the lemmas~\ref{aigle1} and~\ref{aigle2}, which hold for all of them. 
Thus in the following, we drop all the indices $k, \ell, m, \epsilon_0, \epsilon_1, \epsilon_2$. Similarly, when using the cut-off functions of Lemma~\ref{aigle2}, we will of course choose the ones which are adapted to the phase at hand, but we will not keep track of the indices $k, \ell, m,\epsilon_0, \epsilon_1, \epsilon_2$ which they carry. We also drop inessential constants, and forget about the summation in~(\ref{pingouin}) for simplicity. Finally, we choose $\epsilon_0=+$, which is of course harmless. 

Duhamel's formula now reads
$$
u(t) = e^{it\<D\>} u_0 + \int_0^t e^{i(t-s)\<D\>} \frac{u(s)}{\<D\>}\frac{u(s)}{\<D\>}\,ds,
$$
or equivalently, for the profile $f$ in Fourier space,
$$
\widehat{f}(t,\xi) = \widehat{f_0}(\xi) + \widehat{F} (\xi,t)
$$
where we set
$$
\widehat{F} (\xi,t) = \int_0^t \int e^{i s \phi(\xi,\eta)} \frac{\widehat{f} (\eta)}{\<\eta\>} \frac{ \widehat{f} (\xi-\eta)}{\< \xi-\eta \>} \, d\eta\,ds.
$$
Since the data satisfy the hypothesis of the theorem, the estimate
$$
\left\| e^{it\<D\>} u_0 \right\|_X \lesssim \epsilon
$$
is easy. 
Suppose we can prove the a priori estimate
\begin{equation}
\label{bergeronnette}
\left\|e^{it\<D\>}F\right\|_X \lesssim \|u\|_X^2.
\end{equation}
It implies the inequality
$$
\|u\|_X \lesssim \epsilon + \|u\|_X^2,
$$
which gives, after choosing $\epsilon$ small enough, the desired inequality $\|u\|_X\lesssim \epsilon$.

\bigskip

We will actually prove (in Section~\ref{secapriori}) the a priori estimate
$$
\|e^{it\<D\>} G\|_X \lesssim \|u\|_X^2,
$$ 
where $G$ is defined for $t \geq 1$ by
$$
\widehat{G} (\xi,t) = \int_0^t \int e^{i s \phi(\xi,\eta)} \frac{\widehat{f} (\eta)}{\<\eta\>} \frac{ \widehat{f} (\xi-\eta)}{\< \xi-\eta \>} \, d\eta\,ds.
$$
It will spare a lot of cumbersome notations to work on $G$ rather than on $F$, and furthermore it is simple to see how the estimate on $F$ follows from the one on $G$.

\section{Proof of the a priori estimate $\|e^{it\<D\>} G\|_X \lesssim \|u\|_X^2$}

\label{secapriori}

\subsection{Choosing the small constants $\delta_1$, $\delta_2$, $\delta_3$ and the large constant $N$}

Two constants which appear in the definition of the $X$ space, namely $\delta_1$ and $N$, have not been defined yet. We will soon introduce two more: $\delta_2$ and $\delta_3$. 
The former will give the rate at which $\mathcal{R}$ is cut off: namely the cut-off functions around $\mathcal{R}$ shall be of the form $\chi^{t^{-{\delta_2}}}_\mathcal{R}$ $\chi^{t^{-{\delta_2}}}_\mathcal{R}$ $\chi^{t^{-{\delta_2}}}_\mathcal{R}$.
The latter will give the rate at which high frequencies are cut off: the high/low cut-off will be $\theta\left( \frac{\cdot}{t^{\delta_3}} \right)$.

We choose these constants so that
$$
1 >> \delta_2 >> \delta_1 >> \delta_3 > 0 \;\;\;\; \mbox{and} \;\;\;\; N \delta_3 >> 1.
$$
This choice of the constants will ensure that the inequalities~(\ref{chevalier1})~(\ref{chevalier2})~(\ref{chevalier3})~(\ref{chevalier4}) ~(\ref{chevalier5})~(\ref{chevalier6})~(\ref{chevalier7})~(\ref{chevalier8})~(\ref{chevalier9})~(\ref{chevalier10})~(\ref{chevalier11})~(\ref{chevalier12}) hold, which will enable us to close the estimates.

\subsection{Estimates on $e^{it\<D\>} \partial_t f$}

Since $e^{it\<D\>} \partial_t f$ can be written as a sum of terms of the form $u^2$,
$$
\left\| e^{it\<D\>} \partial_t f \right\|_{\left( \frac{2}{3} - 2\delta_1 \right)^{-1}} \lesssim \left\|u\right\|_{\left( \frac{1}{3} - \delta_1 \right)^{-1}}^2 \lesssim \|u\|_X^2 \frac{1}{\<t\>^{1+6\delta_1}}.
$$
Similarly,
$$
\left\| e^{it\<D\>} \partial_t f\right\|_{3/2} \lesssim \|u\|_X^2 \frac{1}{\<t\>}.
$$

\subsection{Estimate for $G$ in the norm $\sup_t \left\|e^{it\<D\>} G \right\|_{H^N}$}

\label{subsecscat}

Using the Strichatz estimate~(\ref{mesange}) and the product law of Lemma~\ref{heron} gives
\begin{equation*}
\begin{split}
\left\| \int_1^t \int e^{is \<D\>} \frac{u}{\<D\>} \frac{u}{\<D\>}\,ds \right\|_{H^N} & \lesssim \left\| \frac{u}{\<D\>} \frac{u}{\<D\>} \right\|_{L^{\left( \frac{1}{2} + \frac{3}{2}\delta_1 \right)^{-1} } W^{N+\frac{5}{6},\left( \frac{5}{6}-\delta_1\right)^{-1}} } \\
& \lesssim \left\| \left\| u \right\|_{H^N} \left\| u \right\|_{L^{\left(\frac{1}{3}-\delta_1 \right)^{-1}}} \right\|_{L^{\left( \frac{1}{2} + \frac{3}{2}\delta_1 \right)^{-1} }} \\
& \lesssim \left\| t^{-\left( \frac{1}{2} + 3\delta_1 \right)} \right\|_{L^{\left( \frac{1}{2} + \frac{3}{2}\delta_1 \right)^{-1} }} \|u\|_X^2 \lesssim \|u\|_X^2.
\end{split}
\end{equation*}

\subsection{Estimate for $G$ in the norm $\sup_t \frac{1}{\sqrt{t}} \left\||x| G \right\|_2$}

By Plancherel theorem, estimating $xG$ in $L^2$ is equivalent to estimating $\partial_{\xi} \widehat{G}(\xi)$ in $L^2$. Thus we want to bound
\begin{subequations}
\begin{align}
\label{canari1}
\partial_\xi \widehat{G}(\xi) = & \int_1^t \int e^{i s \phi} s \partial_\xi \phi \frac{\widehat{f} (\eta)}{\<\eta\>} \frac{ \widehat{f} (\xi-\eta)}{\< \xi-\eta \>} \, d\eta\,ds \\
\label{canari2}
& + \int_1^t \int e^{i s \phi} \frac{\widehat{f} (\eta)}{\<\eta\>} \frac{ \widehat{xf} (\xi-\eta)}{\< \xi-\eta \>} \, d\eta\,ds \\
& \;\;\;\;\;\;\;\;\;\;\;+ \mbox{\{easier term\}},
\end{align}
\end{subequations}
where the ``easier term'' correspond to the case where $\partial_\xi$ hits $\frac{1}{\< \xi-\eta \>}$. In order to estimate~(\ref{canari2}) use successively the Strichartz estimates~(\ref{mesange}), the product law Lemma~\ref{heron} and the estimates~(\ref{canard4}) and~(\ref{canard6}) to obtain
\begin{equation*}
\begin{split}
\left\|(\ref{canari2})\right\|_2 & = \left\| \int_1^t e^{is\<D\>} \frac{u}{\<D\>} \frac{e^{is\<D\>} xf}{\<D\>}\,ds \right\|_2 \\
& \lesssim \left\| \frac{u}{\<D\>} \frac{e^{is\<D\>} xf}{\<D\>} \right\|_{L^2\left([1,t], W^{11/12,6/5}\right)} \\
& \lesssim \left\| \left\|u\right\|_3 \left\| xf \right\|_2 \right\|_{L^2\left([1,t]\right)} \\
& \lesssim \|u\|_X^2 \| 1 \|_{L^2\left([1,t]\right)} \lesssim \sqrt{t} \|u\|_X^2.
\end{split}
\end{equation*}

Next we focus on~(\ref{canari1}). This term will be treated by splitting the frequency space. Writing $u = \chi_{\mathcal{O}}(D) u + \widetilde{\chi}_{\mathcal{O}}(D)u$ gives
\begin{subequations}
\begin{align}
\label{colibri1}
(\ref{canari1}) = & \int_1^t \int e^{i s \phi} s \partial_\xi \phi \frac{\chi_{\mathcal{O}}(\eta) \widehat{f} (\eta)}{\<\eta\>} \frac{ \chi_{\mathcal{O}}(\xi-\eta) \widehat{f} (\xi-\eta)}{\< \xi-\eta \>} \, d\eta\,ds \\
\label{colibri2}
& + \int_1^t \int e^{i s \phi} s \partial_\xi \phi \frac{\widetilde{\chi}_{\mathcal{O}}(\eta) \widehat{f} (\eta)}{\<\eta\>} \frac{\chi_{\mathcal{O}}(\xi-\eta) \widehat{f} (\xi-\eta)}{\< \xi-\eta \>} \, d\eta\,ds \\
& \;\;\;\;\;\;\;\;\;\;\;\;\;\;\;\;\;\;\;\;\;\;\;\;\;\;\;\;\;\;\;\;\;\; + \mbox{\{symmetric and easier terms\}}.
\end{align}
\end{subequations}
(in particular, the interaction of $\widetilde{\chi}_{\mathcal{O}}(D) u$ with itself is easier since the estimates on $\widetilde{\chi}_{\mathcal{O}}(D) u$ are stronger). In order to estimate~(\ref{colibri2}), observe that $\partial_\xi \phi$ is a harmless sum of bounded Fourier multipliers, thus
\begin{equation*}
\begin{split}
\|(\ref{colibri2})\|_2 & = \left\| \int_1^t s e^{is\<D\>} T_{\partial_\xi \phi} \left( \frac{\widetilde{\chi}_{\mathcal{O}}(D)}{\<D\>} u\,,\,\frac{\chi_{\mathcal{O}}(D)}{\<D\>}u\right)\,ds \right\|_2 \\
& \lesssim  \int_1^t s \left\|\frac{\chi_{\mathcal{O}}(D)}{\<D\>} u\right\|_{L^{\left(\frac{1}{3}-\delta_1\right)^{-1}}} \left\|\frac{\widetilde{\chi}_{\mathcal{O}}(D)}{\<D\>} u\right\|_{L^{\left(\frac{1}{6}+\delta_1\right)^{-1}}}\,ds \\
&  \lesssim \|u\|_X^2 \int_1^t s \frac{1}{s^{\frac{1}{2}+3\delta_1}} \frac{1}{s^{1-3\delta_1}}\,ds \lesssim \sqrt{t} \|u\|_X^2.
\end{split}
\end{equation*}
We are left with~(\ref{colibri1}). 
We use here that resonances are separated: since the above term corresponds to interactions for which $\eta$ and $\xi-\eta$ are within $\delta_0$ of $\mathcal{O}$, there cannot be any space-time resonance, which means: either $(\xi,\eta)$ is not in $\mathcal{S}$, or it is not in $\mathcal{T}$. 

In other words, by point 3 of Lemma~\ref{aigle2}, we can split the integration domain of~(\ref{colibri1}) by adding the cut-off functions $\chi_{\mathcal{T}}$ and $\chi_{\mathcal{S}}$, which do not depend on time.
\begin{subequations}
\begin{align}
\label{moineau1}
(\ref{colibri1}) = & \int_1^t \int \chi_\mathcal{T}(\xi,\eta) e^{i s \phi} s \partial_\xi \phi \frac{\chi_{\mathcal{O}}(\eta) \widehat{f} (\eta)}{\<\eta\>} \frac{ \chi_{\mathcal{O}}(\xi-\eta) \widehat{f} (\xi-\eta)}{\< \xi-\eta \>} \, d\eta\,ds \\
\label{moineau2}
& + \int_1^t \int \chi_\mathcal{S}(\xi,\eta) e^{i s \phi} s \partial_\xi \phi \frac{\chi_{\mathcal{O}}(\eta) \widehat{f} (\eta)}{\<\eta\>} \frac{ \chi_{\mathcal{O}}(\xi-\eta) \widehat{f} (\xi-\eta)}{\< \xi-\eta \>} \, d\eta\,ds .
\end{align}
\end{subequations}
Estimating~(\ref{moineau1}) is simple: using that $\partial_\eta \phi$ does not vanish on its integration domain, we integrate by parts with the help of the formula $\frac{\partial_\eta \phi}{is|\partial_\eta \phi|^2} \cdot \partial_\eta e^{is\phi} = e^{is\phi}$ and obtain a term which is very similar to~(\ref{canari2}), and can be estimated in an identical way.

In order to estimate~(\ref{moineau2}), using that $\phi$ does not vanish on its integration domain, integrate by parts via the formula $\frac{1}{i\phi} \partial_s e^{is\phi} = e^{is\phi}$. This gives
\begin{subequations}
\begin{align}
\label{bleue1}
(\ref{moineau2}) = & \int \frac{\chi_\mathcal{S}(\xi,\eta) \partial_\xi \phi}{i\phi} e^{i t \phi} t  \frac{\chi_{\mathcal{O}}(\eta) \widehat{f} (\eta)}{\<\eta\>} \frac{ \chi_{\mathcal{O}}(\xi-\eta) \widehat{f} (\xi-\eta)}{\< \xi-\eta \>} \, d\eta \\
\label{bleue2}
& - \int_1^t \int \frac{\chi_\mathcal{S}(\xi,\eta) \partial_\xi \phi}{i\phi} e^{i s \phi} s  \frac{\chi_{\mathcal{O}}(\eta) \partial_s \widehat{f} (\eta)}{\<\eta\>} \frac{ \chi_{\mathcal{O}}(\xi-\eta) \widehat{f} (\xi-\eta)}{\< \xi-\eta \>} \, d\eta\,ds \\
& \;\;\;\;\;\;\;\;\;\;\;\;\;\;\;\;\;\;\;\;\;\;\;\;\;\;\;\;\;\;\;\;\;\;+ \mbox{\{symmetric and easier terms\}}
\end{align}
\end{subequations}
(where the terms which are not explicitly written are the boundary term at $s=1$, and the terms where $\partial_s$ hits the second $f$, or $s$).
Proceed in a straightforward fashion to estimate~(\ref{bleue1}): by Bernstein's inequality~(\ref{lemmadeltaj}) and Proposition~\ref{oiecendree},
\begin{equation*}
\begin{split}
\|(\ref{bleue1})\|_2 & = t \left\| T_{\frac{\chi_\mathcal{S}(\xi,\eta) \partial_\xi \phi\chi_{\mathcal{O}}(\eta)\chi_{\mathcal{O}}(\xi-\eta) }{i\phi\<\eta\>\< \xi-\eta \>}} (u,u) \right\|_2 \lesssim t \left\| T_{\frac{\chi_\mathcal{S}(\xi,\eta) \partial_\xi \phi\chi_{\mathcal{O}}(\eta)\chi_{\mathcal{O}}(\xi-\eta) }{i\phi\<\eta\>\< \xi-\eta \>}} (u,u) \right\|_{3/2} \\
& \lesssim t \|u\|_3^2 \lesssim \|u\|_X^2.
\end{split}
\end{equation*}
As for~(\ref{bleue2}), still by Bernstein's inequality~(\ref{lemmadeltaj}) and Proposition~\ref{oiecendree},
\begin{equation*}
\begin{split}
\|(\ref{bleue2})\|_2 & = \left\| \int_1^t s e^{is\<D\>} T_{\frac{\chi_\mathcal{S}(\xi,\eta) \partial_\xi \phi\chi_{\mathcal{O}}(\eta)\chi_{\mathcal{O}}(\xi-\eta) }{i\phi\<\eta\>\< \xi-\eta \>}} \left( e^{\pm is\<D\>} \partial_s f , u\right) \,ds \right\|_2 \\
& \lesssim \int_1^t s \left\| T_{\frac{\chi_\mathcal{S}(\xi,\eta) \partial_\xi \phi\chi_{\mathcal{O}}(\eta)\chi_{\mathcal{O}}(\xi-\eta) }{i\phi\<\eta\>\< \xi-\eta \>}} \left( e^{\pm is\<D\>} \partial_s f , u\right) \,ds \right\|_{1} \\
& \lesssim \int_1^t s \|u\|_3 \left\|e^{\pm is\<D\>} \partial_s f\right\|_{3/2} \,ds\\
& \lesssim \|u\|_X^3 \int_1^t s \frac{1}{s^{3/2}} \,ds \lesssim \|u\|_X^3 \sqrt{t}.
\end{split}
\end{equation*}

\subsection{Estimate for $G$ in the norm $\sup_t t^{1-3\delta_1} \left\|\widetilde{\chi}_\mathcal{O}(D) e^{it\<D\>} G \right\|_{\left(\frac{1}{6}+\delta_1 \right)^{-1}}$}

\label{secestG}

Notice first that, since resonances are separated, the localization $\widetilde{\chi}_\mathcal{O}$ effectively cuts off space-time resonances. Thus the frequency space $(\xi,\eta)$ will be split into three parts: large frequencies (larger than $s^{\delta_3}$); small frequencies away from $\mathcal{T}$; small frequencies away from $\mathcal{S}$. More precisely, using $\theta$ defined in~(\ref{colvert})
\begin{subequations}
\begin{align}
\label{grebe1}
\widetilde{\chi}_\mathcal{O}(\xi) \widehat{G}(\xi) = & \int_1^t  \int \widetilde{\chi}_\mathcal{O}(\xi) \left[ 1 - \theta\left(\frac{(\xi,\eta)}{s^{\delta_3}}\right) \right]  e^{i s \phi} \frac{\widehat{f} (\eta)}{\<\eta\>} \frac{ \widehat{f} (\xi-\eta)}{\< \xi-\eta \>} \, d\eta\,ds \\
\label{grebe2}
& + \int_1^t  \int \widetilde{\chi}_\mathcal{O}(\xi) \chi_{\mathcal{S}}(\xi,\eta) \theta\left(\frac{(\xi,\eta)}{s^{\delta_3}}\right)e^{i s \phi} \frac{\widehat{f} (\eta)}{\<\eta\>} \frac{ \widehat{f} (\xi-\eta)}{\< \xi-\eta \>} \, d\eta\,ds \\
\label{grebe3}
& + \int_1^t  \int \widetilde{\chi}_\mathcal{O}(\xi) \chi_{\mathcal{T}}(\xi,\eta) \theta\left(\frac{(\xi,\eta)}{s^{\delta_3}}\right)e^{i s \phi} \frac{\widehat{f} (\eta)}{\<\eta\>} \frac{ \widehat{f} (\xi-\eta)}{\< \xi-\eta \>} \, d\eta\,ds 
\end{align}
\end{subequations}

\subsubsection{The high frequency term~(\ref{grebe1}).}
\label{merle}
It suffices in order to bound it to use the strong control on high frequencies. This is done by using the Littlewood-Paley decomposition, which is recalled in Section~\ref{LP}.
\begin{equation*}
\begin{split}
\mathcal{F}^{-1} (\ref{grebe1}) & = \int_1^t e^{\pm is\<D\>} T_{\frac{\widetilde{\chi}_\mathcal{O}(\xi) \left[ 1 - \theta\left(\frac{(\xi,\eta)}{s^{\delta_3}}\right) \right]}{\<\eta\>\< \xi-\eta \>}} (u,u) \,ds \\
& = \int_1^t \sum_{2^j \geq s^{\delta_3}} e^{\pm is\<D\>} P_{<j+3} T_{\frac{\widetilde{\chi}_\mathcal{O}(\xi) \left[ 1 - \theta\left(\frac{(\xi,\eta)}{s^{\delta_3}}\right) \right]}{\<\eta\>\< \xi-\eta \>}} (P_j u,P_{<j} u) \,ds \mbox{+ \{symmetric term\}}.
\end{split}
\end{equation*}
We now forget about the symmetric term, and use the fact that, by Proposition~\ref{oiecendree}, the elementary pseudo-product operators above are bounded for exponents satisfying the H\"older relation. Using in addition the dispersive estimate~(\ref{grue}) and the inequality~(\ref{deltajsobolev}) gives

\begin{equation*}
\begin{split}
& \left\| \int_1^t  \sum_{2^j \geq s^{\delta_3}} e^{\pm i(t-s)\<D\>}P_{<j+3} T_{\frac{\widetilde{\chi}_\mathcal{O}(\xi) \left[ 1 - \theta\left(\frac{(\xi,\eta)}{s^{\delta_3}}\right) \right]}{\<\eta\>\< \xi-\eta \>}} (P_j u,P_{<j} u) \,ds \right\|_{\left( \frac{1}{6} + \delta_1 \right)^{-1}} \\
& \;\;\;\;\;\;\;\;\;\;\;\;\;\;\; \lesssim \int_1^t \sum_{2^j\gtrsim s^{\delta_3}} \frac{1}{(t-s)^{1 - 3\delta_1}} \left\| P_{<j+3} T_{\frac{\widetilde{\chi}_\mathcal{O}(\xi) \left[ 1 - \theta\left(\frac{(\xi,\eta)}{s^{\delta_3}}\right) \right]}{\<\eta\>\< \xi-\eta \>}} (P_j u,P_{<j} u) \right\|_{W^{2,\left( \frac{5}{6} - \delta_1 \right)^{-1}}} \,ds \\
& \;\;\;\;\;\;\;\;\;\;\;\;\;\;\; \lesssim \int_1^t \sum_{2^j\gtrsim s^{\delta_3}}  \frac{1}{(t-s)^{1 - 3\delta_1}} 2^{2j} \left\| P_j u \right\|_2 \left\| P_{<j} u \right\|_{\left(\frac{1}{3} - \delta_1\right)^{-1}} \,ds\\
& \;\;\;\;\;\;\;\;\;\;\;\;\;\;\; \lesssim \int_1^t \sum_{2^j \gtrsim s^{\delta_3}} \frac{1}{(t-s)^{1 - 3\delta_1}} 2^{2j} 2^{-jN} \left\| u \right\|_{H^N} \left\| u \right\|_{\left(\frac{1}{3} - \delta_1\right)^{-1}} \,ds \\
& \;\;\;\;\;\;\;\;\;\;\;\;\;\;\; \lesssim \int_1^t  \frac{1}{(t-s)^{1 - 3\delta_1}} s^{\delta_3\left(2-N\right)} \frac{1}{s^{\frac{1}{2}+3\delta_1}}\,ds \lesssim \frac{1}{t^{1-3\delta_1}} \|u\|_X^2,
\end{split}
\end{equation*}
where the last inequality holds since
\begin{equation}
\label{chevalier1}
\delta_3 \left(N-2\right) + \frac{1}{2} + 3 \delta_1 > 1.
\end{equation}

\subsubsection{The term away from $\mathcal{T}$ (\ref{grebe2}).} In order to deal with this term, integrate by parts in time using the identity $\frac{1}{i\phi} \partial_s e^{is\phi} = e^{is\phi}$. Thus
\begin{subequations}
\begin{align}
\label{hirondelle1}
(\ref{grebe2}) = & \int \widetilde{\chi}_\mathcal{O}(\xi) \chi_{\mathcal{S}}(\xi,\eta) \theta\left(\frac{(\xi,\eta)}{t^{\delta_3}}\right) \frac{1}{i\phi} e^{i t \phi} \frac{\widehat{f} (\eta)}{\<\eta\>} \frac{ \widehat{f} (\xi-\eta)}{\< \xi-\eta \>} \, d\eta\\
\label{hirondelle2}
& - \int_1^t \int \widetilde{\chi}_\mathcal{O}(\xi) \chi_{\mathcal{S}}(\xi,\eta) \theta\left(\frac{(\xi,\eta)}{s^{\delta_3}}\right) \frac{1}{i\phi} e^{i s \phi}  \frac{\partial_s \widehat{f} (\eta)}{\<\eta\>} \frac{ \widehat{f} (\xi-\eta)}{\< \xi-\eta \>} \, d\eta\,ds \\
& \;\;\;\;\;\;\;\;\;\;\;\;\;\;\;\;\;\;\;\;\;\;\mbox{+ \{symmetric and easier terms\}}, 
\end{align}
\end{subequations}
where the ``symmetric and easier terms'' correspond to the cases where the partial derivative $\partial_s$ hits either the other $f$, or the cut-off function $ \theta\left(\frac{(\xi,\eta)}{s^{\delta_3}}\right)$; and to the boundary term at $s=1$. Using successively Bernstein's inequality~(\ref{lemmadeltaj}) and Proposition~\ref{toucan} gives
\begin{equation*}
\begin{split}
\left\| e^{it\<D\>} \mathcal{F}^{-1} (\ref{hirondelle1}) \right\|_{\left( \frac{1}{6} + \delta_1 \right)^{-1}} & = \left\|T_{\frac{\widetilde{\chi}_\mathcal{O}(\xi) \chi_{\mathcal{S}}(\xi,\eta) \theta\left(\frac{(\xi,\eta)}{s^{\delta_3}}\right)}{i\phi\<\eta\>\<\xi-\eta\>}} (u,u) \right\|_{\left( \frac{1}{6} + \delta_1 \right)^{-1}} \\
& \lesssim t^{\delta_3\left( \frac{3}{2} - 3\delta_1 \right)} \left\| T_{\frac{\widetilde{\chi}_\mathcal{O}(\xi) \chi_{\mathcal{S}}(\xi,\eta) \theta\left(\frac{(\xi,\eta)}{s^{\delta_3}}\right)}{i\phi\<\eta\>\<\xi-\eta\>}} (u,u) \right\|_{\left( \frac{2}{3} - 2 \delta_1 \right)^{-1}} \\
& \lesssim t^{\delta_3\left( \frac{3}{2} - 3\delta_1\right)} t^{A\delta_3} \|u\|_{\left( \frac{1}{3} - \delta_1 \right)^{-1}}^2 \\
& \lesssim t^{\delta_3\left( \frac{3}{2} - 3\delta_1\right)} t^{A\delta_3} \frac{1}{t^{1+6\delta_1}} \|u\|_X^2 \\
& \lesssim \frac{1}{t^{1-3\delta_1}} \|u\|_X^2,
\end{split}
\end{equation*}
where the last inequality holds since
\begin{equation}
\label{chevalier2}
9 \delta_1 > \delta_3 \left( A + \frac{3}{2} - 3\delta_1 \right).
\end{equation}
In order to estimate~(\ref{hirondelle2}), use successively the dispersive estimate~(\ref{grue}), the inequality~(\ref{deltajsobolev}), Bernstein's inequality~(\ref{lemmadeltaj}), and Proposition~\ref{toucan} to get
\begin{equation*}
\begin{split}
& \left\| e^{it\<D\>} \mathcal{F}^{-1} (\ref{hirondelle2}) \right\|_{\left( \frac{1}{6} + \delta_1 \right)^{-1}} = \left\|\int_1^t e^{i(t-s)\<D\>} T_{\frac{\widetilde{\chi}_\mathcal{O}(\xi) \chi_{\mathcal{S}}(\xi,\eta) \theta\left(\frac{(\xi,\eta)}{s^{\delta_3}}\right)}{i\phi\<\eta\>\<\xi-\eta\>}} \left(e^{\pm is\<D\>} (\partial_s f),u\right)\,ds \right\|_{\left( \frac{1}{6} + \delta_1 \right)^{-1}}  \\
& \;\;\;\;\;\;\;\;\;\;\;\;\; \lesssim \int_1^t \frac{1}{(t-s)^{1 - 3\delta_1}} s^{2 \delta_3} \left\| T_{\frac{\widetilde{\chi}_\mathcal{O}(\xi) \chi_{\mathcal{S}}(\xi,\eta) \theta\left(\frac{(\xi,\eta)}{s^{\delta_3}}\right)}{i\phi\<\eta\>\<\xi-\eta\>}} \left(e^{\pm is\<D\>} (\partial_s f),u\right) \right\|_{\left(\frac{5}{6}-\delta_1 \right)^{-1}} \,ds\\
& \;\;\;\;\;\;\;\;\;\;\;\;\; \lesssim \int_1^t \frac{1}{(t-s)^{1 - 3\delta_1}} s^{2 \delta_3} s^{\delta_3 \left( \frac{1}{6}-2\delta_1 \right)} \left\| T_{\frac{\widetilde{\chi}_\mathcal{O}(\xi) \chi_{\mathcal{S}}(\xi,\eta) \theta\left(\frac{(\xi,\eta)}{s^{\delta_3}}\right)}{i\phi\<\eta\>\<\xi-\eta\>}} \left(e^{\pm is\<D\>} (\partial_s f),u\right) \right\|_{\left(1-3\delta_1 \right)^{-1}} \,ds\\
& \;\;\;\;\;\;\;\;\;\;\;\;\; \lesssim \int_1^t \frac{1}{(t-s)^{1 - 3\delta_1}} s^{2 \delta_3} s^{\delta_3 \left(\frac{1}{6}-2\delta_1\right)} s^{A\delta_3} \left\| e^{\pm is\<D\>} (\partial_s f) \right\|_{\left(\frac{2}{3} - 2 \delta_1 \right)^{-1}} \left\| u \right\|_{\left(\frac{1}{3} - \delta_1 \right)^{-1}} \,ds\\
& \;\;\;\;\;\;\;\;\;\;\;\;\; \lesssim \int_1^t \frac{1}{(t-s)^{1 - 3\delta_1}} s^{2 \delta_3} s^{\delta_3 \left(\frac{1}{6}-2\delta_1\right)} s^{A\delta_3} \frac{1}{s^{1+6\delta_1}} \frac{1}{s^{\frac{1}{2}+3 \delta_1}} \|u\|_X^2 \,ds \\
& \;\;\;\;\;\;\;\;\;\;\;\;\; \lesssim \|u\|_X^2 \frac{1}{t^{1-3\delta_1}},
\end{split}
\end{equation*}
where the last inequality holds since
\begin{equation}
\label{chevalier3}
\frac{1}{2} + 9 \delta_1 > \delta_3 \left( A + \frac{13}{6} - 2\delta_1 \right).
\end{equation}

\subsubsection{The term away from $\mathcal{S}$ (\ref{grebe3})} First transform this term by an integration by parts using the identity
$\frac{\partial_\eta \phi}{is |\partial_\eta \phi|^2}\cdot\partial_\eta e^{is\phi} = e^{is\phi}$. This gives
\begin{equation*}
\begin{split}
(\ref{grebe3}) = \int_1^t  \int \widetilde{\chi}_\mathcal{O}(\xi) & \chi_{\mathcal{T}}(\xi,\eta) \theta\left(\frac{(\xi,\eta)}{s^{\delta_3}}\right) \frac{\partial_\eta \phi}{is |\partial_\eta \phi|^2}  e^{i s \phi} \frac{\partial_\eta \widehat{f} (\eta)}{\<\eta\>} \frac{ \widehat{f} (\xi-\eta)}{\< \xi-\eta \>} \, d\eta\,ds \\ 
& \mbox{ + \{symmetric and easier terms\}}.
\end{split}
\end{equation*}
or
\begin{equation*}
\begin{split}
e^{it\<D\>} \mathcal{F}^{-1} (\ref{grebe3}) = \int_1^t e^{i(t-s)\<D\>} & \frac{1}{s} T_{\frac{\widetilde{\chi}_\mathcal{O}(\xi) \chi_{\mathcal{T}}(\xi,\eta) \theta\left(\frac{(\xi,\eta)}{s^{\delta_3}}\right)\partial_\eta \phi}{|\partial_\eta \phi|^2 \<\eta\> \<\xi-\eta \>}} \left( e^{\pm is \<D\>} (xf),u\right) \,ds \\
& \mbox{ + \{symmetric and easier terms\}}.
\end{split}
\end{equation*}
With the help of the dispersive estimate~(\ref{grue}), the inequality~(\ref{deltajsobolev}) and the Proposition~\ref{toucan}, it can be estimated as follows:
\begin{equation*}
\begin{split}
& \left\| \int_1^t e^{i(t-s)\<D\>} \frac{1}{s} T_{\frac{\widetilde{\chi}_\mathcal{O}(\xi) \chi_{\mathcal{T}}(\xi,\eta) \theta\left(\frac{(\xi,\eta)}{s^{\delta_3}}\right)\partial_\eta \phi}{|\partial_\eta \phi|^2\<\eta\> \<\xi-\eta \>}} \left( e^{\pm is \<D\>} (xf),u\right) \,ds \right\|_{\left( \frac{1}{6} + \delta_1 \right)^{-1}} \\
& \;\;\;\;\;\;\;\;\;\;\;\;\;\;\;\;\;\;\lesssim \int_1^t  \frac{1}{(t-s)^{1-3\delta_1}} \frac{1}{s}s^{2 \delta_3} \left\| T_{\frac{\widetilde{\chi}_\mathcal{O}(\xi) \chi_{\mathcal{T}}(\xi,\eta) \theta\left(\frac{(\xi,\eta)}{s^{\delta_3}}\right)\partial_\eta \phi}{|\partial_\eta \phi|^2\<\eta\> \<\xi-\eta \>}} \left( e^{\pm is \<D\>} (xf),u\right) \right\|_{\left( \frac{5}{6} - \delta_1 \right)^{-1}}\,ds\\
& \;\;\;\;\;\;\;\;\;\;\;\;\;\;\;\;\;\;\lesssim \int_1^t \frac{1}{(t-s)^{1-3\delta_1}} \frac{1}{s} s^{2 \delta_3} s^{A\delta_3} \left\|xf\right\|_2 \left\|u\right\|_{\left(\frac{1}{3}-\delta_1 \right)^{-1}}\,ds\\
& \;\;\;\;\;\;\;\;\;\;\;\;\;\;\;\;\;\;\lesssim \int_1^t \frac{1}{(t-s)^{1-3\delta_1}} \frac{1}{s} s^{2 \delta_3} s^{A\delta_3} \|u\|_X^2 \sqrt{s} \frac{1}{s^{\frac{1}{2}+3\delta_1}} \,ds \\
& \;\;\;\;\;\;\;\;\;\;\;\;\;\;\;\;\;\;\lesssim \|u\|_X^2 \frac{1}{t^{1-3\delta_1}},
\end{split}
\end{equation*}
where the last inequality is true since
\begin{equation}
\label{chevalier4}
3\delta_1 > \delta_3 \left( A + 2 \right) > 0.
\end{equation}

\subsection{Estimate for $G$ in the norm $\sup_t t^{\frac{1}{2}+3\delta_1} \left\|e^{it\<D\>} G \right\|_{\left(\frac{1}{3}-\delta_1 \right)^{-1}}$}

First cut off within a distance of order $\delta_0$ of the space-time resonant set by writing
\begin{subequations}
\begin{align}
\label{macareux1}
\widehat{G}(\xi) = & \int_0^t \int \chi_{\mathcal{R}}^{1}(\xi,\eta) e^{i s \phi} \frac{\widehat{f} (\eta)}{\<\eta\>} \frac{ \widehat{f} (\xi-\eta)}{\< \xi-\eta \>} \, d\eta\,ds \\
\label{macareux2}
& + \int_0^t \int \left[ 1 - \chi_{\mathcal{R}}^{1}(\xi,\eta) \right] e^{i s \phi} \frac{ \widehat{f} (\eta)}{\<\eta\>} \frac{ \widehat{f} (\xi-\eta)}{\< \xi-\eta \>} \, d\eta\,ds.
\end{align}
\end{subequations}
The term~(\ref{macareux1}) corresponds to frequencies away from $\mathcal{R}$; it can be dealt with exactly as in Section~\ref{secestG}, even yielding a stronger estimate than needed. Thus it suffices to estimate~(\ref{macareux2}).

Since resonances are separated, it is possible to add Fourier multipliers $\widetilde{\chi}_\mathcal{O}(D)$ to the arguments of~(\ref{canard2}). Splitting furthermore the frequency space $(\xi,\eta)$, we get
\begin{subequations}
\begin{align}
\label{cormoran1}
(\ref{macareux2}) & = \int_1^t \int \chi_\mathcal{R}^{s^{-\delta_2}}(\xi,\eta) e^{i s \phi} \frac{\widetilde{\chi}_\mathcal{O}(\eta) \widehat{f} (\eta)}{\<\eta\>} \frac{\widetilde{\chi}_\mathcal{O}(\xi-\eta) \widehat{f} (\xi-\eta)}{\< \xi-\eta \>} \, d\eta\,ds \\
\label{cormoran2}
& + \int_1^t \int \chi_\mathcal{R}^{1}(\xi,\eta) \chi_\mathcal{S}^{s^{-\delta_2}}(\xi,\eta) \theta\left( \frac{(\xi,\eta)}{s^{\delta_3}} \right) e^{i s \phi} \frac{\widetilde{\chi}_\mathcal{O}(\eta) \widehat{f} (\eta)}{\<\eta\>} \frac{\widetilde{\chi}_\mathcal{O}(\xi-\eta) \widehat{f} (\xi-\eta)}{\< \xi-\eta \>} \, d\eta\,ds \\
\label{cormoran3}
& + \int_1^t \int \chi_\mathcal{R}^{1}(\xi,\eta)  \chi_\mathcal{T}^{s^{-\delta_2}}(\xi,\eta)   e^{i s \phi} \frac{\widetilde{\chi}_\mathcal{O}(\eta) \widehat{f} (\eta)}{\<\eta\>} \frac{\widetilde{\chi}_\mathcal{O}(\xi-\eta)  \widehat{f} (\xi-\eta)}{\< \xi-\eta \>} \, d\eta\,ds.
\end{align}
\end{subequations}

\subsubsection{The term close to $\mathcal{R}$~(\ref{cormoran1})}

In order to estimate this term, recall first that (see Lemma~\ref{aigle2}) 
$$
\chi_\mathcal{R}^{t^{-\delta_2}}(\xi,\eta) = \chi \left( t^{\delta_2} \left(|\eta|-R \right) \right) \chi \left( t^{\delta_2} \left(\xi - \lambda \eta \right) \right).
$$
Thus, it follows with the help of the dispersive estimate~(\ref{grue}) and Proposition~\ref{toucan} that
\begin{equation*}
\begin{split}
& \left\| e^{it\<D\>} \mathcal{F}^{-1} (\ref{cormoran1}) \right\|_{\left(\frac{1}{3}-\delta_1\right)^{-1}} \\
& \;\;\;\;\;\;\; = \left\| \int_1^t e^{i(t-s)\<D\>} T_{\frac{ \chi \left(s^{\delta_2} \left(\xi - \lambda \eta \right) \right)} {\<\eta\>\< \xi-\eta \>}} 
\left( e^{\pm is\<D\>} \chi\left( s^{\delta_2} \left( |D|-R \right)\right) \widetilde{\chi}_\mathcal{O}(D) f,\widetilde{\chi}_\mathcal{O}(D)u \right) \,ds \right\|_{\left(\frac{1}{3}-\delta_1\right)^{-1}} \\
& \;\;\;\;\;\;\;\lesssim \int_1^t \frac{1}{(t-s)^{\frac{1}{2}+3\delta_1}} \left\| \widetilde{\chi}_\mathcal{O}(D) \chi \left( s^{\delta_2} \left( |D|-R \right) \right) f \right\|_2 \left\|\widetilde{\chi}_\mathcal{O}(D) u \right\|_{\left(\frac{1}{6}+\delta_1 \right)^{-1}}\,ds \\
& \;\;\;\;\;\;\;\lesssim \int_1^t \frac{1}{(t-s)^{\frac{1}{2}+3\delta_1}} \frac{1}{s^{\delta_2/24}} \left\||x|^{1/8} \widetilde{\chi}_\mathcal{O}(D) f \right\|_2 \|u\|_X \frac{1}{s^{1-3\delta_1}} \,ds \\
& \;\;\;\;\;\;\;\lesssim  \|u\|_X^2 \int_1^t  \frac{1}{(t-s)^{\frac{1}{2}+3\delta_1}} \frac{1}{s^{\delta_2/24}} \frac{1}{s^{1-3\delta_1}}\,ds\\
& \;\;\;\;\;\;\;\lesssim \|u\|_X^2 \frac{1}{t^{\frac{1}{2}+3\delta_1}}
\end{split}
\end{equation*}
where the last inequality holds since
\begin{equation}
\label{chevalier5}
\frac{\delta_2}{24} > 3\delta_1.
\end{equation}

\subsubsection{The term away from $\mathcal{T}$~(\ref{cormoran2})} Integrate this term by parts via the identity $\frac{1}{i\phi} \partial_s e^{is\phi} = e^{is\phi}$ to get
\begin{subequations}
\begin{align}
\label{pinson1}
(\ref{cormoran2}) & = \int \chi_\mathcal{R}^{1}(\xi,\eta) \chi_\mathcal{S}^{t^{-\delta_2}}(\xi,\eta) e^{i t \phi} \frac{1}{i\phi} \frac{\widetilde{\chi}_\mathcal{O}(\eta) \widehat{f} (\eta)}{\<\eta\>} \frac{\widetilde{\chi}_\mathcal{O}(\xi-\eta) \widehat{f} (\xi-\eta)}{\< \xi-\eta \>} \, d\eta\\
\label{pinson2}
& - \int_1^t \int \chi_\mathcal{R}^{1}(\xi,\eta) \chi_\mathcal{S}^{s^{-\delta_2}}(\xi,\eta) e^{i s \phi} \frac{1}{i\phi} \frac{\widetilde{\chi}_\mathcal{O}(\eta) \partial_s \widehat{f} (\eta)}{\<\eta\>} \frac{\widetilde{\chi}_\mathcal{O}(\xi-\eta) \widehat{f} (\xi-\eta)}{\< \xi-\eta \>} \, d\eta\,ds\\
& \mbox{\;\;\;\;\;\;\;\;\;\;\;\;\;\;\;\;\;\;+ \{symmetric and easier terms \}}.
\end{align}
\end{subequations}
The first term can be estimated with the help of Bernstein's inequality~(\ref{lemmadeltaj}) and Proposition~\ref{toucan}:
\begin{equation*}
\begin{split}
\left\| e^{it\<D\>} \mathcal{F}^{-1} (\ref{pinson1}) \right\|_{\left(\frac{1}{3}-\delta_1\right)^{-1}} & = \left\| T_{ \frac{\chi_\mathcal{R}^{1}(\xi,\eta) \chi_\mathcal{S}^{t^{-\delta_2}}(\xi,\eta)}  {i\phi\<\eta\>\< \xi-\eta \>} } (\widetilde{\chi}_\mathcal{O}(D) u,\widetilde{\chi}_\mathcal{O}(D)u) \right\|_{\left(\frac{1}{3}-\delta_1\right)^{-1}} \\
& \lesssim \left\| T_{ \frac{\chi_\mathcal{R}^{1}(\xi,\eta) \chi_\mathcal{S}^{t^{-\delta_2}}(\xi,\eta)} {i\phi\<\eta\>\< \xi-\eta \>} } (\widetilde{\chi}_\mathcal{O}(D) u,\widetilde{\chi}_\mathcal{O}(D)u) \right\|_{2} \\
& \lesssim  t^{A\delta_2} \left\| \widetilde{\chi}_\mathcal{O}(D) u \right\|_4^2 \\
& \lesssim \|u\|_X^2  t^{A\delta_2} \frac{1}{t^{3/2}} \lesssim \|u\|_X^2 \frac{1}{t^{\frac{1}{2}+3\delta_1}},
\end{split}
\end{equation*}
where the last inequality holds since
\begin{equation}
\label{chevalier6}
3\delta_1 + A \delta_2 < 1.
\end{equation}
For the second term, use successively the dispersive estimate~(\ref{grue}), Bernstein's inequality~(\ref{lemmadeltaj}) and Proposition~\ref{toucan} to get
\begin{equation*}
\begin{split}
& \left\| e^{it\<D\>} \mathcal{F}^{-1} (\ref{pinson1}) \right\|_{\left(\frac{1}{3}-\delta_1\right)^{-1}} \\
& \;\;\;\;\;\;\;\;\;\;\;\;\;\;= \left\| \int_1^t e^{i(t-s)\<D\>} T_{ \frac{\chi_\mathcal{R}^{1}(\xi,\eta) \chi_\mathcal{S}^{s^{-\delta_2}}(\xi,\eta)} {i\phi\<\eta\>\< \xi-\eta \>} } (\widetilde{\chi}_\mathcal{O}(D) e^{\pm i s\<D\>} \partial_s f ,\widetilde{\chi}_\mathcal{O}(D)u) \right\|_{\left(\frac{1}{3}-\delta_1\right)^{-1}} \\
& \;\;\;\;\;\;\;\;\;\;\;\;\;\;\lesssim \int_1^t \frac{1} {(t-s)^{\frac{1}{2}+3\delta_1}} \left\| T_{ \frac{\chi_\mathcal{R}^{1}(\xi,\eta) \chi_\mathcal{S}^{s^{-\delta_2}}(\xi,\eta)} {i\phi\<\eta\>\< \xi-\eta \>} } (\widetilde{\chi}_\mathcal{O}(D) e^{\pm i s\<D\>} \partial_s f ,\widetilde{\chi}_\mathcal{O}(D)u) \right\|_1 \\
& \;\;\;\;\;\;\;\;\;\;\;\;\;\;\lesssim \int_1^t \frac{1}{(t-s)^{\frac{1}{2}+3\delta_1}} s^{A\delta_2} \left\| e^{\pm i s\<D\>} \partial_s f \right\|_{3/2} \left\| u \right\|_3\,ds \\
& \;\;\;\;\;\;\;\;\;\;\;\;\;\;\lesssim \|u\|_X^3 \int_1^t \frac{1}{(t-s)^{\frac{1}{2}+3\delta_1}} s^{A\delta_2} \frac{1}{s} \frac{1}{\sqrt{s}} \,ds \\
& \;\;\;\;\;\;\;\;\;\;\;\;\;\; \lesssim \|u\|_X^3 \frac{1}{t^{\frac{1}{2}+3\delta_1}},
\end{split}
\end{equation*}
where the last inequality is justified since
\begin{equation}
\label{chevalier7}
A\delta_2< \frac{1}{2}.
\end{equation}

\subsubsection{The term away from $\mathcal{S}$~(\ref{cormoran3})}Integrate this term by parts via the identity $\frac{\partial_\eta \phi}{is|\partial_\eta \phi|^2} \cdot \partial_\eta e^{is\phi} = e^{is\phi}$ to get
\begin{subequations}
\begin{align}
& (\ref{cormoran3}) \\ 
\label{pelican1}
& = - \int_1^t \int \chi_\mathcal{R}^{1}(\xi,\eta) \chi_\mathcal{T}^{s^{-\delta_2}}(\xi,\eta) \frac{ \partial_\eta \phi}{is|\partial_\eta \phi|^2} e^{i s \phi} \frac{\widetilde{\chi}_\mathcal{O}(\eta) \widehat{\partial_\eta f} (\eta)}{\<\eta\>} \frac{\widetilde{\chi}_\mathcal{O}(\xi-\eta) \widehat{f} (\xi-\eta)}{\< \xi-\eta \>} \, d\eta\,ds\\
\label{pelican2}
& \;\;\;\;\;\;\;\;\;\;\;\;\;\;\;\;\;\;\;\;\;\mbox{+ \{symmetric and easier terms\}}.
\end{align}
\end{subequations}
Next, using successively the dispersive estimate~(\ref{grue}), inequality~(\ref{deltajsobolev}) and Proposition~\ref{toucan} gives
\begin{equation*}
\begin{split}
& \left\|e^{it\<D\>} \mathcal{F}^{-1} (\ref{pelican1}) \right\|_{\left(\frac{1}{3}-\delta_1\right)^{-1}} \\
& \;\;\;\;\;\; = \left\| \int_1^t e^{i(t-s)\<D\>} \frac{1}{s} T_{\frac{\chi_\mathcal{R}^{1}(\xi,\eta) \chi_\mathcal{T}^{s^{-\delta_2}}(\xi,\eta) \partial_\eta \phi}{i|\partial_\eta \phi|^2\<\eta\>\< \xi-\eta \>}} 
\left( e^{\pm is\<D\> }\widetilde{\chi}_\mathcal{O}(D)(xf) , \widetilde{\chi}_\mathcal{O}(D)u  \right)\,ds \right\|_{\left(\frac{1}{3}-\delta_1\right)^{-1}} \\
& \;\;\;\;\;\;\lesssim \int_1^t \frac{1}{s} \frac{1} {(t-s)^{\frac{1}{2}+3\delta_1}} \left\| T_{\frac{\chi_\mathcal{R}^{1}(\xi,\eta) \chi_\mathcal{T}^{s^{-\delta_2}}(\xi,\eta) \partial_\eta \phi}{i|\partial_\eta \phi|^2\<\eta\>\< \xi-\eta \>}} \left( e^{\pm is\<D\>} \widetilde{\chi}_\mathcal{O}(D)(xf) , \widetilde{\chi}_\mathcal{O}(D)u  \right) \right\|_{\left( \frac{2}{3}+\delta_1\right)^{-1}} \,ds\\
& \;\;\;\;\;\;\lesssim \int_1^t \frac{1}{s} \frac{1}{(t-s)^{\frac{1}{2}+3\delta_1}} s^{A \delta_2} \left\|xf\right\|_2 \left\| \widetilde{\chi}_\mathcal{O}(D)u \right\|_{\left(\frac{1}{6} + \delta_1\right)^{-1}}\,ds \\
& \;\;\;\;\;\;\lesssim \|u\|_X^2 \int_1^t \frac{1}{s} \frac{1}{(t-s)^{\frac{1}{2}+3\delta_1}} s^{A \delta_2} \sqrt{s} \frac{1}{s^{1-3\delta_1}}\,ds \\
& \;\;\;\;\;\;\lesssim \|u\|_X^2 \frac{1}{t^{\frac{1}{2}+3\delta_1}},
\end{split}
\end{equation*}
where the last inequality is valid since
\begin{equation}
\label{chevalier8}
A\delta_2 + 3 \delta_1 < 1.
\end{equation}

\subsection{Estimate for $G$ in the norm $\sup_t \left\||x|^{1/8} \widetilde{\chi}_\mathcal{O} G \right\|_2$}

First notice that this estimate is far from being optimal (it seems likely that the best possible estimate is $\sup_t \left\||x|^{1/2-\epsilon} G \right\|_2 < \infty$ for $\epsilon>0$). In other words, we have a lot of room to our disposal, and we will perform very crude estimates, which simplifies some technical points.

Also observe that interpolating between~(\ref{canard1}) and~(\ref{canard4}) gives
$$
\||x|^{1/8} f \|_2 \lesssim t^{1/16}.
$$

As usual, we start by decomposing $G$ as follows
\begin{subequations}
\begin{align}
\label{sterne1}
\widetilde{\chi}_\mathcal{O}(\xi) \widehat{G}(\xi) = & \int_1^t  \int \widetilde{\chi}_\mathcal{O}(\xi) \left[ 1 - \theta\left(\frac{(\xi,\eta)}{s^{\delta_3}}\right) \right]  e^{i s \phi} \frac{\widehat{f} (\eta)}{\<\eta\>} \frac{ \widehat{f} (\xi-\eta)}{\< \xi-\eta \>} \, d\eta\,ds \\
\label{sterne2}
& + \int_1^t  \int \widetilde{\chi}_\mathcal{O}(\xi) \chi_{\mathcal{S}}(\xi,\eta) \theta\left(\frac{(\xi,\eta)}{s^{\delta_3}}\right)e^{i s \phi} \frac{\widehat{f} (\eta)}{\<\eta\>} \frac{ \widehat{f} (\xi-\eta)}{\< \xi-\eta \>} \, d\eta\,ds \\
\label{sterne3}
& + \int_1^t  \int \widetilde{\chi}_\mathcal{O}(\xi) \chi_{\mathcal{T}}(\xi,\eta) \theta\left(\frac{(\xi,\eta)}{s^{\delta_3}}\right)e^{i s \phi} \frac{\widehat{f} (\eta)}{\<\eta\>} \frac{ \widehat{f} (\xi-\eta)}{\< \xi-\eta \>} \, d\eta\,ds 
\end{align}
\end{subequations}

\subsubsection{The high frequency term~(\ref{sterne1})} 
First use the Littlewood-Paley decomposition recalled in Section~\ref{LP} to write
\begin{equation*}
\begin{split}
\mathcal{F}^{-1} (\ref{sterne1}) & = \int_1^t e^{\pm is\<D\>} T_{\frac{\widetilde{\chi}_\mathcal{O}(\xi) \left[ 1 - \theta\left(\frac{(\xi,\eta)}{s^{\delta_3}}\right) \right]}{\<\eta\>\< \xi-\eta \>}} (u,u) \,ds \\
& = \int_1^t \sum_{2^j \gtrsim s^{\delta_3} \log s} e^{\pm is\<D\>} P_{<j+3} T_{\frac{\widetilde{\chi}_\mathcal{O}(\xi) \left[ 1 - \theta\left(\frac{(\xi,\eta)}{s^{\delta_3}}\right) \right]}{\<\eta\>\< \xi-\eta \>}} (P_j u,P_{<j} u) \,ds \mbox{+ \{symmetric term\}}.
\end{split}
\end{equation*}
Next, we forget about the symmetric term, and use successively Lemma~\ref{kiwi}, Proposition~\ref{toucan}, Lemma~\ref{kiwi} again and finally the inequality~(\ref{deltajsobolev}) to get
\begin{equation*}
\begin{split}
& \left\||x|^{1/8} \int_1^t \sum_{2^j \gtrsim s^{\delta_3} } e^{\pm is\<D\>} P_{<j+3} T_{\frac{\chi_\mathcal{O}(\xi) \left[ 1 - \theta\left(\frac{(\xi,\eta)}{s^{\delta_3}}\right) \right]}{\<\eta\>\< \xi-\eta \>}} (P_j u,P_{<j} u) \,ds \right\|_{2} \\
& \;\;\;\;\;\;\;\;\;\;\;\;\;\;\;\;\;\;\;\;\;\;\;\;\;\lesssim \int_1^t  \sum_{2^j \gtrsim s^{\delta_3}} s^{1/8} \left\| |x|^{1/8} P_{<j+3} T_{\frac{\widetilde{\chi}_\mathcal{O}(\xi) \left[ 1 - \theta\left(\frac{(\xi,\eta)}{s^{\delta_3}}\right) \right]}{\<\eta\>\< \xi-\eta \>}} (P_j u,P_{<j} u) \right\|_2 \,ds \\
& \;\;\;\;\;\;\;\;\;\;\;\;\;\;\;\;\;\;\;\;\;\;\;\;\;\lesssim \int_1^t \sum_{2^j \gtrsim s^{\delta_3}} s^{1/8} \left\| |x|^{1/8} P_{<j} u \right\|_2 \|P_j u\|_\infty\,ds \\
& \;\;\;\;\;\;\;\;\;\;\;\;\;\;\;\;\;\;\;\;\;\;\;\;\;\lesssim \int_1^t \sum_{2^j \gtrsim s^{\delta_3}} s^{1/8} s^{1/8} \left\||x|^{1/8} f \right\|_2 2^{3j/2} 2^{-jN} \|u\|_{H^N} \,ds\\
& \;\;\;\;\;\;\;\;\;\;\;\;\;\;\;\;\;\;\;\;\;\;\;\;\;\lesssim \|u\|_X^2 \int_1^t s^{1/8} s^{1/8} s^{\delta_3\left(\frac{3}{2}-N \right)} s^{1/16}\,ds \|u\|_{H^N} \\
& \;\;\;\;\;\;\;\;\;\;\;\;\;\;\;\;\;\;\;\;\;\;\;\;\;\lesssim \|u\|_X^2 ,
\end{split}
\end{equation*}
where the last inequality holds true since
\begin{equation}
\label{chevalier9}
\delta_3 \left(N-\frac{3}{2} \right) > \frac{21}{16}.
\end{equation}

\subsubsection{The term away from $\mathcal{T}$~(\ref{sterne2})}
Integrating by parts with the help of the identity $\frac{1}{i\phi} \partial_s e^{is\phi} = e^{is\phi}$ gives
\begin{subequations}
\begin{align}
\label{perruche1}
(\ref{sterne2}) & = \int \widetilde{\chi}_\mathcal{O}(\xi) \chi_{\mathcal{S}}(\xi,\eta) \theta\left(\frac{(\xi,\eta)}{t^{\delta_3}}\right) \frac{1}{i\phi} e^{i t \phi} \frac{\widehat{f} (\eta)}{\<\eta\>} \frac{ \widehat{f} (\xi-\eta)}{\< \xi-\eta \>} \, d\eta\\
\label{perruche2}
& - \int_1^t \int \widetilde{\chi}_\mathcal{O}(\xi) \chi_{\mathcal{S}}(\xi,\eta) \theta\left(\frac{(\xi,\eta)}{s^{\delta_3}}\right) \frac{1}{i\phi} e^{i t \phi}  \frac{\partial_s \widehat{f} (\eta)}{\<\eta\>} \frac{ \widehat{f} (\xi-\eta)}{\< \xi-\eta \>} \, d\eta\,ds \\
& \mbox{+ \{symmetric and easier terms\}}.
\end{align}
\end{subequations}
In order to treat~(\ref{perruche1}), observe that the arguments of the pseudo-product have frequency of order less than $t^{\delta_3}$. Thus it is possible to add to them a Fourier multiplier $P_{<\delta_3 \log t + C}$ for a constant $C$ (see Section~\ref{LP} for the definition of the projections $P$.
Using Proposition~\ref{toucan}, Bernstein's inequality~(\ref{lemmadeltaj}) and Lemma~\ref{kiwi} yields the desired bound for (\ref{perruche1}):
\begin{equation*}
\begin{split}
\left\| |x|^{1/8} \mathcal{F}^{-1} (\ref{perruche1}) \right\|_2 & = \left\| |x|^{1/8} T_{\frac{\widetilde{\chi}_\mathcal{O}(\xi) \chi_{\mathcal{S}}(\xi,\eta) \theta\left(\frac{(\xi,\eta)}{t^{\delta_3}}\right)} {i\phi\<\eta\>\<\xi-\eta\>}} (u,P_{<\delta_3 \log t + C}u) \right\|_2 \\
& \lesssim t^{A\delta_3} \left\| \<x\>^{1/8} u \right\|_2 \left\|P_{<\delta_3 \log t + C} u \right\|_\infty \\ 
& \lesssim t^{A\delta_3} t^{1/8} \left\| \<x\>^{1/8} f \right\|_2 t^{\delta_3} \left\| u \right\|_3 \\
& \lesssim \|u\|_X^2 t^{A\delta_3} t^{1/8} t^{1/8} t^{\delta_3} \frac{1}{\sqrt{t}} \\
&  \lesssim \|u\|_X^2,
\end{split}
\end{equation*}
where the last inequality holds since
\begin{equation}
\label{chevalier10}
\delta_3 (A+1) < \frac{5}{16}.
\end{equation}
As far as~(\ref{perruche2}) is concerned, it is still possible to add a Fourier multiplier $P_{<\delta_3 \log t + C}$  to the arguments of the pseudo-product. Using successively Lemma~\ref{kiwi}, Proposition~\ref{toucan}, Bernstein's inequality~\ref{lemmadeltaj}, and once again Lemma~\ref{kiwi} gives
\begin{equation*}
\begin{split}
& \left\| |x|^{1/8} \mathcal{F}^{-1} (\ref{perruche1}) \right\|_2 = \left\| |x|^{1/8} \int_1^t T_{\frac{\widetilde{\chi}_\mathcal{O}(\xi) \chi_{\mathcal{S}}(\xi,\eta) \theta\left(\frac{(\xi,\eta)}{s^{\delta_3}}\right)}{i\phi \<\eta\>\<\xi-\eta\>}} \left(P_{<\delta_3 \log t + C} e^{\pm i s \<D\>} \partial_s f,u \right) \,ds \right\|_2\\
& \;\;\;\;\;\;\;\;\;\;\;\;\;\;\;\;\;\;\;\;\;\lesssim \left\| s^{\delta_3} \left[ \<x\> + s \right]^{1/8} T_{\frac{\widetilde{\chi}_\mathcal{O}(\xi) \chi_{\mathcal{S}}(\xi,\eta) \theta\left(\frac{(\xi,\eta)}{s^{\delta_3}}\right)}{i\phi \<\eta\>\<\xi-\eta\>} } \left(P_{<\delta_3 \log t + C} e^{\pm i s \<D\>} \partial_s f,u \right) \right\|_{L^2([1,t], L^{6/5})} \\
& \;\;\;\;\;\;\;\;\;\;\;\;\;\;\;\;\;\;\;\;\;\lesssim \left\| s^{\delta_3} s^{1/8} s^{A\delta_3} \left\| P_{<\delta_3 \log t + C} e^{\pm i s \<D\>} \partial_s f \right\|_{L^3_x} \|\< x \>^{1/8} u \|_{L^2_x} \right\|_{L^2([1,t])} \\
& \;\;\;\;\;\;\;\;\;\;\;\;\;\;\;\;\;\;\;\;\;\lesssim \left\| s^{\delta_3} s^{1/8} s^{A\delta_3} s^{\delta_3} \|e^{\pm i s \<D\>} \partial_s f\|_{L^{3/2}_t} \| \<x\>^{1/8} f \|_{L^2_x} \right\|_{L^2([1,t])}\\
& \;\;\;\;\;\;\;\;\;\;\;\;\;\;\;\;\;\;\;\;\;\lesssim \|u\|_X^3 \left\| s^{\delta_3} s^{1/8} s^{A\delta_3} s^{\delta_3} \frac{1}{s} s^{1/8} s^{1/16} \right\|_{L^2([1,t])}\\
& \;\;\;\;\;\;\;\;\;\;\;\;\;\;\;\;\;\;\;\;\;\lesssim \|u\|_X^3,
\end{split}
\end{equation*}
where the last inequality is valid since
\begin{equation}
\label{chevalier11}
(A+2) \delta_3 < \frac{3}{16}.
\end{equation}

\subsubsection{The term away from $\mathcal{S}$~(\ref{sterne3})} Integrating by parts with the help of the identity
$\frac{\partial_\eta \phi}{is|\partial_\eta \phi|^2} \cdot \partial_\eta e^{is\phi} = e^{is\phi}$ gives
\begin{subequations}
\begin{align}
& (\ref{sterne3}) \\
\label{autruche1}
& = - \int_1^t \int \widetilde{\chi}_\mathcal{O}(\xi) \chi_\mathcal{T}(\xi,\eta) \theta\left(\frac{(\xi,\eta)}{s^{\delta_3}}\right) \frac{\partial_\eta \phi}{is|\partial_\eta \phi|^2} e^{i s \phi} \frac{ \widehat{\partial_\eta f} (\eta)}{\<\eta\>} \frac{(\xi-\eta) \widehat{f} (\xi-\eta)}{\< \xi-\eta \>} \, d\eta\,ds \\
\label{autruche2}& \;\;\;\;\;\;\;\;\;\;\;\;\;\;\;\;\;\;\;\;\;\mbox{+ \{symmetric and easier terms\}}.
\end{align}
\end{subequations}
Applying successively Lemma~\ref{kiwi}, Bernstein's inequality~(\ref{lemmadeltaj}), Proposition~\ref{toucan}, and once again Lemma~\ref{kiwi} gives
\begin{equation*}
\begin{split}
\left\| |x|^{1/8} \mathcal{F}^{-1} (\ref{autruche1}) \right\|_2 &= \left\| |x|^{1/8} \int_1^t e^{is\<D\>} \frac{1}{s}  T_{\frac{\widetilde{\chi}_\mathcal{O}(\xi) \chi_\mathcal{T}(\xi,\eta) \partial_\eta \phi\theta\left(\frac{(\xi,\eta)}{s^{\delta_3}}\right)}{i|\partial_\eta \phi|^2\<\eta\>\< \xi-\eta \>}} (u,u) \right\|_2 \\
& \lesssim \int_1^t \frac{1}{s}s^{1/8} s^{\delta_3} \left\| |x|^{1/8} T_{\frac{\widetilde{\chi}_\mathcal{O}(\xi) \chi_\mathcal{T}(\xi,\eta) \partial_\eta \phi\theta\left(\frac{(\xi,\eta)}{s^{\delta_3}}\right)}{i|\partial_\eta \phi|^2\<\eta\>\< \xi-\eta \>}} (u,u) \right\|_{6/5}\,ds \\
& \lesssim \int_1^t \frac{1}{s}s^{1/8} s^{\delta_3} s^{A \delta_3} \left\| \<x\>^{1/8} u \right\|_2 \|u\|_3 \,ds \\
& \lesssim \int_1^t \frac{1}{s}s^{1/8} s^{\delta_3} s^{A \delta_3} s^{1/8} \left\|\<x\>^{1/8} f \right\|_2 \|u\|_3 \,ds \\
& \lesssim \|u\|_X^2 \int_1^t  \frac{1}{s} s^{1/8} s^{\delta_3} s^{A \delta_3} s^{1/8} s^{1/16} \frac{1}{\sqrt{s}} \,ds \lesssim \|u\|_X^2,
\end{split}
\end{equation*}
where the last inequality holds since
\begin{equation}
\label{chevalier12}
\delta_3(A+1) < \frac{3}{16}.
\end{equation}

\section{Auxiliary tools}

We will review or prove below a few estimates used in the proof of our main theorem.

\subsection{Littlewood-Paley decomposition}

\label{LP}

Consider $\psi$ a function supported in the annulus $\mathcal{C}(0,\frac{3}{4},\frac{8}{3})$ such that
$$
\mbox{for $\xi \neq 0$,}\;\;\;\;\sum_{j \in \mathbb{Z}} \psi \left( \frac{\xi}{2^j} \right) = 1 .
$$
Define first
$$
\Phi(\xi) \overset{def}{=} \sum_{j <0} \psi \left( \frac{\xi}{2^j} \right)
$$
and then the Fourier multipliers
$$
P_j \overset{def}{=} \psi \left( \frac{D}{2^j} \right) \;\;\;\;\; P_{<j} = \Phi \left( \frac{D}{2^j} \right).
$$
This gives a homogeneous and an inhomogeneous decomposition of the identity (for instance, in $L^2$)
$$
\sum_{j \in \mathbb{Z}} P_j = \operatorname{Id} \;\;\;\;\mbox{and}\;\;\;\;\;P_{<0} + \sum_{j \geq 0} P_j = \operatorname{Id}.
$$
All these operators are bounded on $L^p$ spaces:
$$
\mbox{if $1<p<\infty$,}\;\;\;\; \|P_j f \|_p \lesssim \|f\|_p \;\;\;\;,\;\;\;\; \|P_{<j} f \|_p \lesssim \|f\|_p.
$$
It is easy to see that
\begin{equation}
\label{deltajsobolev}
\mbox{if $j \geq 0$,}\;\;\;\;\|P_j f\|_{W^{s,p}} \sim 2^{js} \|f\|_{p} \;\;\;\;\mbox{and}\;\;\;\;\|P_{<j} f \|_{W^{s,p}} \lesssim 2^{js} \|f\|_{p}.
\end{equation}
Also recall Bernstein's lemma: if $1\leq q\leq p \leq \infty$,
\begin{equation}
\label{lemmadeltaj}
\|P_j f \|_p \leq 2^{3j\left( \frac{1}{q}-\frac{1}{p} \right)} \left\| P_j f \right\|_q\;\;\;\;\;\;\mbox{and}\;\;\;\;\; \left\| P_{<j} f \right\|_p \leq 2^{3j\left( \frac{1}{q}-\frac{1}{p} \right)} \left\| P_{<j} f \right\|_q .
\end{equation}

\subsection{Estimates for the linear Klein-Gordon equation}

We will consider here the linear Klein-Gordon equation for $c=0$, but of course the results presented in this section remain unchanged if $c>0$.

The lack of homogeneity of the Klein-Gordon dispersion relation gives rise to a wide range of dispersive or Strichartz estimates; roughly speaking, it admits ``wave like'' estimates, but also ``Schr\"odinger like'' estimates with a loss of derivatives.

Thus we do not attempt to give complete references, but simply borrow from Ginibre and Velo~\cite{GV} for dispersive estimates and from Ibrahim, Masmoudi and Nakanishi~\cite{IMN} for Strichartz estimates.

The dispersive estimates we will need read
\begin{equation}
\label{grue}
\left\| e^{it\<D\>} f \right\|_p \lesssim t^{\frac{3}{p}-\frac{3}{2}} 
\|f\|_{W^{ 4 \left( \frac{1}{2} - \frac{1}{p} \right) + \epsilon, p'}}
\;\;\;\;\;\;\mbox{if $2\leq p \leq \infty$ and $\epsilon>0$.}
\end{equation}
As for the Strichartz estimates, we state them in a very particular case:
\begin{equation}
\label{mesange}
\left\| \int_0^t e^{is\<D\>} F(s)\,ds \right\|_2 \lesssim \left\| F \right\|_{L^{\left( \frac{1}{2} + \frac{3}{2}\delta \right)^{-1} } W^{\frac{5}{6}-\frac{5}{2}\delta+\epsilon,\left( \frac{5}{6}-\delta \right)^{-1}}}  \;\;\;\;\;\;\mbox{for $\epsilon>0$ and $0\leq \delta \leq \frac{1}{3}$}.
\end{equation}

We now turn to weighted versions of the above:
\begin{lem} 
\label{kiwi}
(i) If $w \geq 0$,
$$
\left\| |x|^w e^{it\<D\>} f \right\|_{2} \lesssim t^w \|f\|_2 + \||x|^w f\|_2.
$$
(ii) If $w,\epsilon>0$, $R\geq 1$, and $\varphi \in \mathcal{C}^\infty_0$,
$$
\left\| |x|^w \varphi\left( \frac{D}{R} \right) \int_0^\infty e^{is\<D\>} F(s) \,ds \right\|_2 \lesssim R^{\frac{5}{6}+\epsilon} \left\| \left[ \< s \> + \<x\> \right]^w F \right\|_{L^2 L^{6/5}}.
$$
(iii) If $w,\epsilon>0$, and $\varphi \in \mathcal{C}^\infty_0$,
$$
\left\| |x|^w  \int_1^\infty \varphi \left( \frac{D}{s^\alpha} \right) e^{is\<D\>} F(s) \,ds \right\|_2 \lesssim \left\| s^{\alpha \frac{5}{6} + \epsilon} \left[ \< s \> + \<x\> \right]^w F(s) \right\|_{L^2 L^{6/5}}.
$$
\end{lem}

\begin{proof} \underline{Assertion (i).}
It is well-known that the kernel $E$ of $e^{it\<D\>}$ is smooth outside of the set $\{|x|=t\}$ and rapidly decaying for $|x|>>|t|$:
\begin{equation}
\label{goeland}
|E(y,t)| \lesssim \frac{1}{\< |y|-|t|| \>_+^N} \;\;\;\;\;\mbox{for any $N$}
\end{equation}
(see for instance H\"ormander~\cite{H} for a closely related statement). In deriving the desired estimate, we will use a smooth cut-off function $\chi$, which is equal to $1$ in $B(0,10)$, and to $0$ in $B(0,20)^c$. It gives the splitting
\begin{equation*}
E(t) * f = \left[\chi \left( \frac{\cdot}{t} \right) E(t)\right] * f + \left[ \left( 1 - \chi \left( \frac{\cdot}{t} \right)\right) E(t) \right]  * f.
\end{equation*}
In order to estimate the first term, observe that the operator with kernel $\chi \left( \frac{\cdot}{t} \right) E(t)$ is uniformly bounded on $L^2$, and has support in $B(0,20t)$. It follows easily (for instance, by decomposing $\mathbb{R}^3$ into coronas with radii of order $nt$, with $n \in \mathbb{N}$) that
$$
\left\|\left[ \chi \left( \frac{\cdot}{t} \right) E(t)\right] * f \right\|_{L^2(\<x\>^{2w}\,dx)} \lesssim t^{w} \left\| f \right\|_{2}.
$$
As for the second term, notice that its norm in $L^2 (\<x\>^{2w}\,dx)$ equals the $L^2$ norm of
$$
h(x) = \int \frac{\<x\>^{w}}{\<y\>^{w}} \left[ 1 - \chi \left( \frac{x-y}{t} \right) \right] E(x-y,t) g(y) \,dy
$$
where $g = \<x\>^{w} f$ is in $L^2$. Due to inequality~(\ref{goeland}), the kernel above can be bounded by
$$
\left| \frac{\<x\>^{w}}{\<y\>^{w}} \left[ 1 - \chi \left( \frac{x-y}{t} \right) \right] E(x-y,t) \right| \lesssim \frac{1}{\< |x-y| \>^{N-w}}
$$
which is in $L^1$ for $N$ big enough. Therefore, $\|h\|_2 \lesssim \|g\|_2$,
hence the desired estimate.

\bigskip \noindent \underline{Assertion (ii).}
It suffices to prove it for $w=0$ and $w=1$, and to interpolate. The case $w=0$ is (\ref{mesange}); for the case $w=1$, use the identity $x e^{it\<D\>_c} f = e^{it\<D\>} \left[ t \frac{D}{\<D\>} f + xf \right]$ to get
\begin{equation*}
\begin{split}
&\left\| x \int_0^\infty e^{is\<D\>} \varphi\left( \frac{D}{R} \right) F(s)\,ds \right\|_2  = \left\|\int_0^\infty e^{is\<D\>} \left[ s \frac{D}{\<D\>}\varphi\left( \frac{D}{R} \right) F(s) + x \varphi\left( \frac{D}{R} \right) F(s) \right] \,ds \right\|_2 \\
& \;\;\;\;\;\;\;\;\;\;\;\;\;\;\;\;\;\lesssim  \left\|s \varphi\left( \frac{D}{R} \right) F(s)\right\|_{L^{2} W^{\frac{5}{6} + \epsilon,6/5}} + \left\|x \varphi\left( \frac{D}{R} \right) F(s)\right\|_{L^{2} W^{\frac{5}{6} + \epsilon,6/5}} \\
& \;\;\;\;\;\;\;\;\;\;\;\;\;\;\;\;\;\lesssim R^{5/6+\epsilon} \left[ \left\|s F(s)\right\|_{L^{2} L^{6/5}} + \left\| \<x\> F (s)\right\|_{L^2 L^{6/5}} \right] \\
& \;\;\;\;\;\;\;\;\;\;\;\;\;\;\;\;\;\lesssim R^{5/6+\epsilon} \left\| \left[ \< s \> + \<x\> \right] F \right\|_{L^2 L^{6/5}}
\end{split}
\end{equation*}
(the first inequality above follows from~(\ref{mesange}) and the boundedness of the operator $ \frac{D}{\<D\>}$ over $L^p$ spaces; in the second, we use $\|\<x\> \varphi(D/T) w\|_p \lesssim \|\<x\> w\|_p$).

\bigskip \noindent \underline{Assertion (iii).} Using the point (ii) that was just proved and the inequality $\|\<x\> \varphi(D/T) w\|_p \lesssim \|\<x\> w\|_p$ gives, taking $C_0$ large enough,
\begin{equation*}
\begin{split}
&\left\| |x|^w \int_0^\infty \varphi \left( \frac{D}{s^\alpha} \right) e^{is\<D\>} F(s) ds \right\|_2  \lesssim \sum_{j\geq 0} \left\| |x|^w P_{<\alpha j + C_0} \int_{2^j}^{2^{j+1}}  e^{is\<D\>} \varphi \left( \frac{D}{s^\alpha} \right)F(s) ds \right\|_2 \\
& \;\;\;\;\;\;\;\;\;\;\;\;\;\;\;\;\;\;\;\;\;\;\lesssim \sum_j 2^{\left( \frac{5}{6}+ \epsilon \right)\alpha j} \left\| \left[ \< s \> + \<x\> \right]^w \varphi \left( \frac{D}{s^\alpha} \right) F(s) \right\|_{L^2 \left( [2^j,2^{j+1}], L^{6/5} \right)} \\
& \;\;\;\;\;\;\;\;\;\;\;\;\;\;\;\;\;\;\;\;\;\;\lesssim \sum_j 2^{\left( \frac{5}{6}+ \epsilon \right)\alpha j} \left\| \left[ \< s \> + \<x\> \right]^w F(s) \right\|_{L^2 \left( [2^j,2^{j+1}], L^{6/5} \right)}\\
& \;\;\;\;\;\;\;\;\;\;\;\;\;\;\;\;\;\;\;\;\;\;\lesssim \left\| s^{\frac{5}{6}\alpha + (\alpha+1) \epsilon} \left[ \< s \> + \<x\> \right]^w F(s) \right\|_{L^2 L^{6/5}}.
\end{split}
\end{equation*}
\end{proof}

\subsection{Pseudo-product operators}

Recall the definition (which was introduced by Coifman and Meyer~\cite{CM}) of the pseudo-product operator with symbol $m(\xi,\eta)$:
$$
T_m(f,g) \overset{def}{=} \mathcal{F}^{-1} \int m(\xi,\eta) \widehat{f}(\eta) \widehat{f}(\xi-\eta) \,d\eta.
$$
The following proposition gives essentially the simplest framework for which boundedness between Lebesgue spaces can be established.

\begin{prop}
\label{oiecendree}
(i) If the Lebesgue exponents $p,q,r$ satisfy the H\"older relation $\frac{1}{p}+\frac{1}{q}=\frac{1}{r}$, then
$$
\left\| T_m(f,g) \right\|_r \lesssim \left\| \widehat{m} \right\|_1 \|f\|_p \|g\|_q.
$$
(where $\widehat{m}$ is the 6-dimensional Fourier transform of $m$).

\medskip
(ii) Still assuming that $\frac{1}{p}+\frac{1}{q}=\frac{1}{r}$, for $w \geq 0$,
$$
|x|^w \lesssim \left\||\alpha|^w \widehat{m}(\alpha,\beta) \right\|_1 \|f\|_p \|g\|_q + \|\widehat{m}\|_1 \|f\|_p \left\||x|^w g(x) \right\|_q .
$$
\end{prop}

\begin{proof} \underline{Assertion (i).} Translating the definition of $T_m$ in physical space yields
$$
T_m(f,g)(x) = \int \int \mu(z-x,y-z) f(y) g(z)\, dy\,dz
$$
with $\mu = \widehat{m}$. The proposition follows from its dual version
$$
\left| \< T_m(f,g)\,,\,h \> \right| \lesssim \|\mu\|_1 \|f\|_p \|g\|_q \|h\|_{r'}.
$$
It suffices of course to prove this under the assumption that $1 = \|f\|_p = \|g\|_q = \|h\|_{r'}$, which is done using Young's inequality:
\begin{equation*}
\begin{split}
\left| \< T_m(f,g)\,,\,h \> \right| & = \left|\int \int \int \mu(z-x,y-z) f(y) g(z) h(x) \, dy\,dz\,dx \right| \\ 
& \lesssim \left| \int \int  \int |\mu(z-x,y-z)| \left[ |f(y)|^p + |g(z)|^q + |h(x)|^{r'} \right]\,dx\,dy\,dz \right| \lesssim \|\mu\|_1.
\end{split}
\end{equation*}

\medskip
\underline{Assertion (ii).} Follows along the same lines from the inequality
\begin{equation*}
\begin{split}
& |x|^w \left| \int \int \mu(z-x,y-z) f(y) g(z) \,dy\,dz\right| \\
& \;\;\;\;\;\;\;\;\;\;\;\;\;\;\;\;\;\;\;\;\;\;\;\;\;\;\;\;\;\;\;\;\;\;\;\;\;\lesssim \int \int |z-x|^w |\mu(z-x,y-z)| |f(y)| |g(z)| \,dy\,dz \\
& \;\;\;\;\;\;\;\;\;\;\;\;\;\;\;\;\;\;\;\;\;\;\;\;\;\;\;\;\;\;\;\;\;\;\;\;\;\;\;\;\;\;\;\;+ \int \int |\mu(z-x,y-z)| |f(y)|  |z|^w |g(z)| \,dy\,dz.
\end{split}
\end{equation*}
\end{proof}

Since the $L^1$ norm of a function is controlled by the $H^{\frac{3}{2}+\epsilon}$ norm of its Fourier transform, one obtains the

\begin{cor}
\label{loriot}
If $p,q,r$ satisfy $\frac{1}{p}+\frac{1}{q}=\frac{1}{r}$, and if $\epsilon>0$ then
$$
\left\| T_m(f,g) \right\|_r \lesssim \left\|m\right\|_{H^{\frac{3}{2}+\epsilon}} \|f\|_p \|g\|_q.
$$
\end{cor}

We also need estimates for pseudo-product operators of the following particular kind:

\begin{prop}
\label{geai}
Let $m(\xi,\eta) = \chi \left( \xi -\lambda \eta \right)$. Then
$$
\left\|T_m(f,g)\right\|_r \lesssim \|\widehat{\chi}\|_1 \|f\|_p \|g\|_q \;\;\;\;\;\;\mbox{if $\frac{1}{p}+\frac{1}{q} = \frac{1}{r}$}.
$$
\end{prop}

\begin{proof}
The dual formulation for $T_m$ can be written as
$$
\< T_m(f,g) \,,\,h\> = \int \int \widehat{\chi}(t) f(y) g(-\lambda t + y) h(-(\lambda+1)t+y) \,dt\,dy.
$$
Proceeding as in the proof of proposition~\ref{oiecendree}, we get under the assumption $\|f\|_p = \|g\|_q = \|h\|_{r'}$ that
$$
\left\| \< T_m(f,g) \,,\,h\> \right\| \lesssim \int |\widehat{\chi}(t)| \left[ |f(y)|^p + |g(-\lambda t + y)|^p +  |h( -(\lambda+1)t+y)|^{r'} \right] \,dt\,dy \lesssim \|\widehat{\chi}\|_1.
$$
\end{proof}

\subsection{Product law}

\begin{lem} If $\frac{1}{p} + \frac{1}{q} = \frac{1}{r}$, $s \geq0$, and $\epsilon>0$, then
\label{heron}
$$\|fg\|_{W^{s,r}} \lesssim \|f\|_{W^{s+\epsilon,p}} \|g\|_q + \|f\|_q \|g\|_{W^{s+\epsilon,p}}.$$
\end{lem}
The proof follows for instance from the paraproduct decomposition; it is very classical, so it will not be included here.

\appendix

\section{Precise study of the resonances}

In this section, we first present the proofs of the lemmas~\ref{aigle1} and~\ref{aigle2}; then, thanks to a numerical computation, we describe the space-time resonant set if $c=5$.

\subsection{Preliminaries}
\label{subsecprel}
Recall the definition of the phases:
$$
\phi^{k,\ell, m}_{\epsilon_0,\epsilon_1,\epsilon_2}(\xi,\eta) \overset{def}{=} \epsilon_0 \< \xi \>_k - \epsilon_1 \< \eta \>_\ell - \epsilon_2 \<\xi-\eta\>_m.
$$
and that (dropping in the following line all indices)
$$
\mathcal{T} = \{ \phi = 0 \}\;\;\;,\;\;\;\mathcal{S} = \{ \partial_\eta \phi =0 \}\;\;\;\mbox{and}\;\;\;\mathcal{R} = \mathcal{S} \cap \mathcal{T}.
$$
In order to establish lemmas~\ref{aigle1} and~\ref{aigle2}, one has to consider a lot of particular cases; but they are all similar and elementary, so we do not detail all of them, and instead only treat one.

Notice that it is actually possible to reduce the number of combinations of the indices to be examined by observing that $\eta$ and $\xi-\eta$ play symmetric roles; and that turning each of the $\epsilon_i$ to its opposite simply turns $\phi$ into $-\phi$.

\subsection{Proof of Lemma~\ref{aigle1}}

\label{secprolem}

We explain it in the case where $c \geq 1$, and for the set $\mathcal{R}^{c1c}_{+--}$, the other possibilities being very similar.

The space-resonant set, $\mathcal{S}^{c1c}_{+--}$, is given by the frequencies $(\xi,\eta)$ such that
$$
0 = \partial_\eta \phi^{c1c}_{+--}(\xi,\eta) = - \frac{\eta}{\<\eta\>} - \frac{c^2(\eta-\xi)}{\<\eta-\xi\>_c}.
$$
This implies that $\eta$ and $\xi-\eta$ are positively colinear, which can be written (so as to be consistent with previous notations) $\xi-\eta = (\lambda-1) \eta$ or $\xi = \lambda \eta$, for a real number $\lambda \geq 1$. A simple computation gives
\begin{equation*}
\lambda(|\eta|) = 1 + \frac{1}{\sqrt{(c^4-c^2)|\eta|^2 + c^4}}.
\end{equation*}
Therefore, setting, for $(r,\omega) \in \mathbb{R}^+ \times \mathbb{S}^{2}$,
$$
p(r,\omega) = (\lambda(r) r \omega\,,\,r \omega) \in \mathbb{R}^6,
$$
$\mathcal{S}^{c1c}_{+--}$ can be parameterized by
$$
\mathcal{S}^{c1c}_{+--} = \{ p(r,\omega)\;,\;\mbox{with $(r,\omega) \in \mathbb{R}^+ \times \mathbb{S}^2$} \}.
$$
A point of $\mathcal{S}^{c1c}_{+--}$, parameterized as above, will belong to $\mathcal{T}^{c1c}_{+--}$, and thus to $\mathcal{R}^{c1c}_{+--}$, if
$$
Z(r) \overset{def}{=} \phi^{c1c}_{+--}(p(r,\omega)) = (\< \lambda(r) r \>_c - \< r \> - \< (\lambda(r)-1)r \>_c = 0.
$$
The function of $r$ on the above right-hand side is analytic and has, as can easily be checked, a non zero limit as $r$ goes to infinity. Therefore its zeroes form a finite set, and each one is of finite order. Suppose that $r_0$ is one of them. 

On the one hand, in a neighbourhood of the component of $\mathcal{R}$ corresponding to $r_0$,
\begin{equation}
\label{petrel1}
\operatorname{dist}(p(r),\mathcal{R}) \sim |r-r_0|.
\end{equation}
On the other hand, $r_0$ is a zero of $Z$ of finite order, thus $|Z(r)| \sim |r-r_0|^m$, for some integer $m$. Since $\phi$ does not vanish to infinite order on $\mathcal{T}$, this implies
\begin{equation}
\label{petrel2}
\operatorname{dist}(p(r),\mathcal{T}) \sim |r-r_0|^{n}.
\end{equation}
for an integer $n$. Combining~(\ref{petrel1}) and~(\ref{petrel2}) gives
\begin{equation}
\label{sittelle}
\operatorname{dist}(p(r),\mathcal{T}) \gtrsim \operatorname{dist}(p(r),\mathcal{R})^{n}
\end{equation}
which is the finite order intersection property.

\subsection{Proof of Lemma~\ref{aigle2}}
\subsubsection{Low frequencies}

In this subsection, we ignore high frequencies and define $\chi_\mathcal{S}^\rho(\xi,\eta)$ and $\chi_\mathcal{S}^\rho(\xi,\eta)$ for $|(\xi,\eta)|\leq M$.

It suffices to do so in the neighbourhood of one of the components of $\mathcal{R}$; it is of the type $\{\xi = \lambda \eta\,,\,|\eta|=R \}$ for real numbers $\lambda,R$. Define
$$
\chi_{\mathcal{R}}^\rho(\xi,\eta) \overset{def}{=} \chi_1 \left( \frac{|\eta|-R}{\rho} \right) \chi_1 \left( \frac{\xi - \lambda \eta}{\rho} \right),
$$
where $\chi_1$ is valued in $[0,1]$, compactly supported, smooth, and equal to $1$ in a neighbourhood of zero. Next define 
$$
\chi_{\mathcal{S}}^\rho (\xi,\eta) \overset{def}{=} \left[ 1 - \chi_{\mathcal{R}}^\rho(\xi,\eta) \right] \chi_2 \left( \frac{ \operatorname{dist}((\xi,\eta),\mathcal{T}) - \operatorname{dist}((\xi,\eta),\mathcal{S})}{\operatorname{dist} ((\xi,\eta),\mathcal{R})^{n+1}} \right)
$$
where $\chi_2$ is valued in $[0,1]$, equal to $0$ on $(-\infty,-1)$ and $1$ on $(1,\infty)$, and where $n$ is the constant appearing in~(\ref{sittelle}). Finally,
$$
\chi_{\mathcal{T}}^\rho (\xi,\eta) \overset{def}{=} 1 - \chi_{\mathcal{R}}^\rho(\xi,\eta) - \chi_{\mathcal{S}}^\rho (\xi,\eta).
$$
The inequality~(\ref{rossignol1}) follows from the above formulas, the finite order intersection property, and the fact that $\phi$ as well as $\partial_\xi \phi$ vanish to finite order on their respective zero sets.

\subsubsection{High frequencies}

In this subsection, we show how to define $\chi_\mathcal{S}^\rho(\xi,\eta)$ and $\chi_\mathcal{T}^\rho(\xi,\eta)$ for $|(\xi,\eta)|\geq M$; recall that $M$ is so large that no space-time resonances belong to this range. 

The problem we are facing occurs when $(\xi,\eta) \rightarrow \infty$ so it suffices to define $\chi_\mathcal{S}^\rho(\xi,\eta)$ and $\chi_\mathcal{T}^\rho(\xi,\eta)$ for $|(\xi,\eta)|$ very large; but then, the index $\rho$ does not have any importance, so we forget about it.

One has to consider all the possible combinations of indices $\epsilon_0,\epsilon_1,\epsilon_2,k,l,m$. Since all these cases can be treated in a similar fashion, we only illustrate here the example of $\phi_{+--}^{11c}$ when $c>1$. As in Subsection~\ref{secprolem}, we find that
$$
\mathcal{S}_{+--}^{11c} = \{ p(r,\omega)\;,\;\mbox{with $(r,\omega) \in \mathbb{R}^+ \times \mathbb{S}^2$} \}
$$
with $p(r,\omega) = (\lambda(r) r \omega\,,\,r \omega)$ and $\lambda(r) = 1 + \frac{1}{\sqrt{(c^4-c^2)r^2 + c^4}}$. A simple limit computation gives that, as $r \rightarrow \infty$,
\begin{equation}
\label{balbuzard}
p(r,\omega) = \left( \left[ r + \frac{1}{\sqrt{c^4-c^2}} \right]\omega\,,\,r\omega \right) + O \left(\frac{1}{r^3}\right)\;\;\;\;\mbox{and}\;\;\;\;\phi_{+--}^{11c}(p(r,\omega)) = C_0 + O \left(\frac{1}{r}\right)
\end{equation}
for a certain constant $C_0$. Define
$$
\chi_\mathcal{S} (\xi,\eta) \overset{def}{=} \chi_3\left(|\xi-\eta| - \frac{1}{\sqrt{c^4-c^2}} \right),
$$
where $\chi_3$ is smooth, supported on a sufficiently small neighbourhood of 0, and equal to 1 near 0.  Since $\partial_{\xi,\eta} \phi$ is bounded, the property~(\ref{balbuzard}) implies that, for $(\xi,\eta)$ large, $|\phi| \gtrsim 1$ on the support of $\chi_\mathcal{S}$. This gives the inequality~(\ref{rossignol2}) for $\frac{\chi_\mathcal{S}}{\phi}$.

Finally, set
$$
\chi_\mathcal{T} (\xi,\eta) \overset{def}{=} 1 - \chi_\mathcal{S} (\xi,\eta).
$$
It can be easily seen that $\partial_\eta \phi \gtrsim 1$ on the support of $\chi_\mathcal{S}$. This gives the inequality~(\ref{rossignol2}) for $\frac{\chi_\mathcal{T} \partial_\eta \phi}{|\partial_\eta \phi|^2}$.

\subsection{Numerical computation for $c=5$}

We computed numerically the space time resonances in case $c=5$. The main interest of this computation is to give a practical example where space-time resonances occur, and are separated.

As showed in Subsection~\ref{secprolem}, finding space-time resonances reduces to finding the zeroes of an explicit real-valued function of a real variable; doing this numerically is very simple.

The results are as follows
\begin{itemize}
\item Space-time resonances only occur for the phases $\phi^{c11}_{+--}$, $\phi^{cc1}_{+--}$, and all the phases obtained from these two by the symmetries signaled in Subsection~\ref{subsecprel}. 
\item Outcome frequencies of space-time resonances are the frequencies $\xi$ such that $|\xi| = 0.3535533906\dots$ or $0.3603654667\dots$.
\item Source frequencies of space-time resonances are the frequencies $\xi$ such that $|\xi| = 0.01314860997\dots$, $0.1767766953\dots$, or $0.3472168567\dots$.
\end{itemize}

\bigskip
\noindent
{\bf Acknowledgements:} The author is grateful to Nader Masmoudi for suggesting this question to him and for early discussions.

\end{document}